\documentclass[11pt,reqno]{amsart}
\newtheorem{theorem}{Theorem}[section]
\newtheorem{proposition}[theorem]{Proposition}

\newtheorem{remark}[theorem]{Remark}
\newtheorem{lemma}[theorem]{Lemma}
\newtheorem{example}[theorem]{Example}
\newtheorem{definition}[theorem]{Definition}

\usepackage{hyperref}
\usepackage{appendix}
\usepackage{amssymb}
\usepackage{amscd}
\usepackage{enumerate}
\usepackage{mathrsfs}
\usepackage{graphicx}
\usepackage{amsmath}
\usepackage{amsfonts}
\usepackage{color}
\usepackage{bm}
\usepackage{cleveref}
\usepackage{mathtools}
\usepackage{leftidx}
\usepackage{forest}
\usepackage{bbm}

\newcommand{\C}{\mathbb{C}}

\usepackage[all]{xy}
\usepackage[latin1]{inputenc} 
\numberwithin{equation}{section}

\allowdisplaybreaks[4]
\begin{document}
	\title{extensions of finite irreducible modules over rank two Lie conformal algebras}
	\author{Lipeng Luo$^{1}$, Yucai Su$^{2}$, Mengjun Wang$^{3}$}
	\address{\textsuperscript{1}Department of Mathematics, Zhejiang University of Science and Technology, Hangzhou, 310023, P.R.China.} 
	\address{\textsuperscript{2}Department of Mathematics, Jimei University, Xiamen, Fujian, 361021, P.R.China.} 
	\address{\textsuperscript{3}School of  Mathematical  Sciences, Nanjing University, Nanjing, 210093, P.R.China.} 
	\email{\textsuperscript{1}luolipeng@zust.edu.cn}
	\email{\textsuperscript{2}ycsu@tongji.edu.cn}
	\email{\textsuperscript{3}wangmj@nju.edu.cn}
	\thanks{This work was supported by the National Natural Science Foundation of China
(Nos. 11971350, 12171129, 12271406, 12471027, 12401035) and the Jiangsu Funding Program for Excellent Postdoctoral Talent (No. 2023ZB063) and the Natural Science Foundation of Shanghai (No. 24ZR1471900).}

	\subjclass[2010]{17B10; 17B65; 17B68}

\keywords{conformal algebra, conformal module, extensions, rank two}
\footnote{The third author is the corresponding author: Mengjun Wang(wangmj@nju.edu.cn)}
\begin{abstract}
In this paper, we give a complete classification of extensions of finite irreducible conformal modules over rank two Lie conformal algebras.
	\end{abstract}
\maketitle
\baselineskip=16pt
	\section{Introduction}
	In the 1990s, V. Kac \cite{dk,k} introduced an axiomatic definition of the Lie conformal algebra, which gives an algebraic description of the singular part of the operator product extansion (OPE) of the chiral fields in 2-dimensional conformal field theory. In addition to being closely related to vertex algebra and conformal field theory, the theory of Lie conformal algebra is also closely associated with infinite-dimensional Lie algebra \cite{bdak}, Hamiltonian formal systems of nonlinear evolution equations \cite{bdk}, and quantum physics \cite{en}, and thus has received more attention in recent years.
	
	A conformal algebra is called finite if it is a finitely generated $\mathbb{C}[\partial]$-module, and the rank of a finite conformal algebra is just its rank as a $\mathbb{C}[\partial]$-module. It was shown in \cite{bch} that a rank two conformal algebra is one of the following four types up to isomorphism:
	\begin{enumerate}[(i)]
		\item a semisimple conformal algebra;
		\item a solvable conformal algebra;
		\item the direct sum of a commutative Lie conformal algebra of rank one and the Virasoro conformal algebra $Vir$ (called in this paper the Lie conformal algebra of Type I);
		\item and what we called the Lie conformal algebra of Type II(see case (2ii) in Proposition \ref{type12}).
	\end{enumerate}
The classification of their finitely irreducible conformal modules can be found in \cite{ck,xhw}. Even for finite semisimple conformal algebras, however, conformal modules of Lie conformal algebras are generally not completely irreducible. Therefore, solving the extension problem plays an essential role in studying the representation theory of conformal algebras. For instance, extensions of finite irreducible conformal modules over the Virasoro, the current, the Neveu-Schwarz and the semi-direct sum of the Virasoro and the current conformal algebras were investigated by Cheng, Kac and Wakimoto in \cite{ckw1,ckw2}, that over supercurrent conformal algebras were classified by Lam in \cite{l}, and that over the Schr\"{o}dinger-Virasoro conformal algebras were considered by Yuan and Ling in \cite{yl}.

In this paper, extensions of finite irreducible modules over rank two conformal algebras are characterized by dealing with certain polynomial equations induced by corresponding module actions. Let $\mathcal{R}$ be a conformal algebra of rank two. If $\mathcal{R}$ is semisimple, then $\mathcal{R}$ is isomorphic to one of the following:$$Cur(\mathfrak{g}),\quad Vir\ltimes Cur(\mathfrak{h}),\quad Vir\oplus Vir,$$ where $Cur(\mathfrak{g})$ (resp. $Cur(\mathfrak{h})$) is the associated current algebra of Lie algebra $\mathfrak{g}$ (resp. $\mathfrak{h}$) with $dim(\mathfrak{g})=2$(resp. $dim(\mathfrak{h})=1$) \cite{dk}. Since extensions of finite irreducible modules over $Cur(\mathfrak{g})$ and $ Vir\ltimes Cur(\mathfrak{h})$ has been given in \cite{ckw1}, we focus on $Vir\oplus Vir$ for the semisimple case and the results are listed in Theorem \ref{t21}, \ref{t22} and \ref{t23}. The nontrivial extensions for solvable rank two conformal algebras can be seen in Theorem \ref{t11}, \ref{t12} and \ref{t13}, while those for non-semisimple and non-solvable Lie conformal algebras of Type I are described in Theorem \ref{t31}, \ref{t32} and \ref{t33}. If $\mathcal{R}$ is a non-semisimple and non-solvable Lie conformal algebra of Type II, one can find a basis $ \{A,B\} $ such that  	\begin{equation*}
[A_\lambda A]=(\partial+2\lambda)A+Q(\partial,\lambda)B,\quad [A_\lambda B]=(\partial+a \lambda +b)B,\quad[B_\lambda B]=0,
\end{equation*}
where $a,b \in \mathbb{C}$ and  $Q(\partial,\lambda)$ is some skew-symmetric polynomial depending on $a,b$, i.e. $Q(\partial,\lambda)=-Q(\partial,-\partial-\lambda)$. If $Q(\partial,\lambda)=\beta(\partial+2\lambda),a=1,b=0,$ $\mathcal{R}$ is the algebra called $\mathcal{L}(\beta)$ in \cite{wx}. If $Q(\partial,\lambda)=0$, $\mathcal{R}$ is just the Lie conformal algebra $\mathcal{W}(a,b)$ whose extension problem has been investigated in \cite{lhw1,lhw2}. Particularly, $\mathcal{W}(1-b,0)$ is the Lie conformal algebra $\mathcal{W}(b)$ in \cite{ly2} and $\mathcal{W}(1,0)$ is just the Heisenberg-Virasoro Lie conformal algebra in \cite{ly1}. So in this case we consider $\mathcal{R}$ under the condition that $Q(\partial,\lambda)\neq0$ and the results can be found in Theorem \ref{t41}, \ref{t42} and \ref{t43}. Fixed $a=1,b=0,Q(\partial,\lambda)=\partial+2\lambda$, our results are consistent with those mentioned in \cite{yw}.

	The paper is organized as follows. In section \ref{s2}, we give basic definitions and related properties of Lie conformal algebras and their modules. In section \ref{s3}, we discuss extensions of finite irreducible modules over a semisimple rank two Lie conformal algebra. In section \ref{s4}, we consider extensions of finite irreducible modules over a solvable rank two Lie conformal algebra. For finite irreducible modules over non-solvable and non-semisimple Lie conformal algebras which have two different types, we give a complete classification of extensions in section \ref{s5} and section \ref{s6} respectively. All vector spaces and tensor products are considered over the complex field $\mathbb{C}$. We use notation $\mathbb{C}^\times$ to represent the set of nonzero complex numbers.

	\section{Preliminaries}\label{s2}
	In this section, we recall the definition of Lie conformal algebras, conformal modules and their extensions, and some known results that are useful in this paper.  For more details, one can refer to \cite{bch,ck,ckw1,k,xhw}.
	\begin{definition}{\label{d1}} A \textbf{Lie conformal algebra} $\mathcal{R}$ is a $\mathbb{C}[\partial]$-module equipped with a $\mathbb{C}$-linear map (called \textbf{$\lambda$-bracket})  $\mathcal{R} \otimes \mathcal{R} \rightarrow \mathcal{R}[\lambda]$, $a\otimes b \mapsto [a_\lambda b]$, satisfying the following axioms
		\begin{equation*}
		\begin{aligned}
		&[\partial a_\lambda b]=-\lambda[a_\lambda b], \   [a_\lambda \partial b]=(\partial+\lambda)[a_\lambda b],\ \ (conformal\   sequilinearity),\\
		&[a_\lambda b]=-[b_{-\lambda-\partial}a]\ \ (skew-symmetry),\\
		&[a_\lambda [b_\mu c]]=[[a_\lambda b]_{\lambda+\mu}c]+[b_\mu[a_\lambda c]]\ \ (Jacobi \ identity),
		\end{aligned}
		\end{equation*}
		for $a$, $b$, $c\in \mathcal{R}$.
	\end{definition}
	
A Lie conformal algebra $\mathcal{R}$ is called \textbf{finite} if $\mathcal{R}$ is finitely generated as a $\mathbb{C}[\partial]$-module. The \textbf{rank} of a finite Lie conformal algebra is just its rank as a $\mathbb{C}[\partial]$-module.

 \begin{definition}{\label{d2}} Let $\mathcal{R}$ be a Lie conformal algebra. A \textbf{conformal $\mathcal{R}$-module} $V$  is a $\mathbb{C}[\partial]$-module endowed with a $\mathbb{C}$-linear map  $\mathcal{R} \otimes V \rightarrow V[\lambda]$, $a\otimes v \mapsto a_\lambda v$, satisfying the following axioms:
	\begin{equation*}
	\begin{aligned}
	&(\partial a)_\lambda v=-\lambda(a_\lambda v), \ \   a_\lambda (\partial v)=(\partial+\lambda)a_\lambda v,\\
	&a_\lambda (b_\mu v)=[a_\lambda b]_{\lambda+\mu}v+b_\mu(a_\lambda v),
	\end{aligned}
	\end{equation*}
	for all $a$, $b\in \mathcal{R}$ and $v\in V$.
\end{definition}
A conformal module $V$ is called \textbf{irreducible} if there is no nonzero submodule $W$ such that $W\neq V$, and $V$ is said to be a trivial $\mathcal{R}$-module if $\mathcal{R}$ acts on $V$ trivially. For any $\eta \in \mathbb{C}$, we can obtain a trivial $\mathcal{R}$-module $\mathbb{C}c_{\eta}=\mathbb{C}$, which is determined by $\eta$, via the action $\partial c_{\eta}=\eta c_{\eta}, \mathcal{R}_\lambda c_{\eta}=0$. It is easy to check that the modules $\mathbb{C}c_{\eta}$ with $\eta \in \mathbb{C}$ exhaust all trivial irreducible $\mathcal{R}$-modules.
\begin{definition}
		Let $V$ and $W$ be two modules over a Lie conformal algebra $\mathcal{R}$. An \textbf{extension} of $W$ by $V$ is an exact sequence of $\mathcal{R}$-modules of the form
		\begin{align}
		0\longrightarrow V \stackrel{i}{\longrightarrow} E \stackrel{p}{\longrightarrow}W \longrightarrow 0,
		\end{align}
		where $E$ is isomorphic to $V\oplus W$ as a vector space.
		Two extensions $0\longrightarrow V \stackrel{i}{\longrightarrow} E \stackrel{p}{\longrightarrow}W \longrightarrow 0$ and $0\longrightarrow V \stackrel{i'}{\longrightarrow} E' \stackrel{p'}{\longrightarrow}W \longrightarrow 0$ are said to be \textbf{equivalent} if there exists a homomorphism of modules such that the following diagram commutes
		\begin{align}
		\begin{CD}
		0 @>>> V   @>i>>   E   @>p>>   W   @>>>0  \\
		@.      @V{1_V}VV       @V{\Psi}VV       @V{1_W}VV  @.    \\
		0 @>>> V   @>i'>>   E'   @>p'>>   W  @>>>0.
		\end{CD}
		\end{align}
				
\end{definition}

Obviously, the direct sum of modules $V\oplus W$ gives rise to an extension $0\to V\to V\oplus W\to W\to 0$.  Any extension $0\to V\to E\to W\to 0 $, which is equivalent to  $0\rightarrow V\to V\oplus W\to W\to 0$, is  called \textbf{trivial extension}.

In general, an extension can be thought of as the direct sum of vector spaces $E=V\oplus W$, where $V$ is a submodule of $E$, while for $w\in W$ we have
\begin{align*}
a_\lambda\cdot w=a_\lambda w+f_{a_\lambda}(w),\   a\in \mathcal{R},
\end{align*}
where $f_{a_\lambda}: W\to \mathbb{C}[\lambda]\otimes V$ is a linear map satisfying the \textbf{cocycle} condition:
\begin{align*}
f_{[a_\lambda b]_{\lambda+\mu}}(w)=f_{a_\lambda}(b_\mu w)+a_\lambda f_{b_\mu}(w)-f_{b_\mu}(a_\lambda w)-b_\mu f_{a_\lambda}(w),\   b\in \mathcal{R}.
\end{align*}
The set of all cocycles forms a vector space $\mathcal{E}xt(W,V)$ over $\mathbb{C}$. Cocycles equivalent to trivial extensions are called \textbf{coboundaries}. They form a subspace $\mathcal{E}xt^c(W,V)$  and the quotient space $\mathcal{E}xt(W,V)/\mathcal{E}xt^c(W,V)$ is denoted by $Ext(W,V)$. The dimension of the quotient space is called the \textbf{dimension} of the space of extensions of $W$ by $V$, denoted by $dim(Ext(W,V))$. From now on, we will say two extensions are \textbf{equivalent} if they belong to the same equivalent class unless confusion is possible.

\begin{example}
	Let $\mathcal{R}$ be an arbitrary conformal algebra, and we can consider extensions of trivial irreducible $\mathcal{R}$-modules of the form
	\begin{align}\label{f00}
	0\longrightarrow \mathbb{C}c_{\eta} \longrightarrow  E\longrightarrow  \mathbb{C}c_{\bar{\eta}} \longrightarrow   0.
	\end{align}
	In this case, $E$ as a vector space is isomorphic to $\mathbb{C}c_{\eta}\oplus \mathbb{C}c_{\bar{\eta}}$, where $\mathbb{C}c_{\eta}$ is a $\mathcal{R}$-submodule, and the following identities hold in $E$:
	\begin{equation}
	R_\lambda c_{\bar{\eta}}=f_R(\lambda)c_{\eta},\quad \partial c_{\bar{\eta}}=\bar{\eta}c_{\bar{\eta}}+tc_{\eta},\label{ex0}
	\end{equation}
	for any $R\in\mathcal{R}$, where $\eta, \bar{\eta}, t \in \mathbb{C}$ and $f_R(\lambda)$ is some polynomial depending on $R$. Since $E$ is a $\mathcal{R}$-module, it follows from $R_\lambda (\partial c_{\bar{\eta}})=(\partial+\lambda)R_\lambda c_{\bar{\eta}}$ that $$\bar{\eta}f_R(\lambda)=(\eta+\lambda) f_R(\lambda)$$ which implies $f_R(\lambda)=0$ for any $R\in\mathcal{R}$. Assume that (\ref{f00}) is a trivial extension, that is, there exists $c_\sigma=kc_\eta+lc_{\bar{\eta}}$, where $k,l\in\mathbb{C}$ and $l\neq0$, such that $$\partial c_{\sigma}=\bar{\eta} c_{\sigma}=\bar{\eta}kc_\eta+\bar{\eta}lc_{\bar{\eta}},\quad \mathcal{R}_\lambda c_{\sigma}=0.$$On the other hand, it follows from (\ref{ex0}) that $$\partial c_{\sigma}=k\partial c_\eta+l\partial c_{\bar{\eta}}=(\eta k+tl)c_\eta+\bar{\eta}lc_{\bar{\eta}}.$$ So we have $(\bar{\eta}-\eta) k=tl$. 
	
	If $\bar{\eta}\neq\eta$, for arbitrary $t$, we can find such $k,l\in\mathbb{C}$, and $E$ is always a trivial extension. If $\bar{\eta}=\eta$, $E$ is trivial only when $t=0$. Therefore, dim($Ext(\mathbb{C}c_{\eta},\mathbb{C}c_{\bar{\eta}})$)=$ \delta_{\eta,\bar{\eta}} $ and when $\eta=\bar{\eta}$, the nontrivial extensions are given (up to equivalence) by $$\mathcal{R}_\lambda c_{\bar{\eta}}=0,\quad \partial c_{\bar{\eta}}=\bar{\eta}c_{\bar{\eta}}+kc_{\eta}$$ with $k\neq0$.
\end{example}

	For convenience, we use $\mathcal{R}$ to denote a free rank two Lie conformal algebra in the sequel. Any finite semisimple Lie conformal algebra, as shown in \cite{dk}, is the direct sum of $Vir,\ Cur(\mathfrak{g}),\ Vir\ltimes Cur(\mathfrak{g})$, where $Cur(\mathfrak{g})$ is the current conformal algebra associated with a finite dimensional semisimple Lie algebra $\mathfrak{g}$. So if $\mathcal{R}$ is semisimple, then $\mathcal{R}$ is isomorphic to either the direct sum of two Virasoro Lie conformal algebras, or $Cur(\mathfrak{g})$ with $ dim(\mathfrak{g})=2 $, or $Vir\ltimes Cur(\mathfrak{h})$ with $dim(\mathfrak{h})=1$. As for the non-semisimple case, we have the following proposition. 
	\begin{proposition}\label{type12}[\cite{bch}, Theorem 3.21]
		Let $\mathcal{R}$ be a rank two  Lie conformal algebra that is not semisimple. 
		\begin{itemize}
			\item[(1)] If $\mathcal{R}$ is solvable, then there is a basis $\{A,B\}$ such that
			\begin{equation}\label{a1}
			[A_\lambda A]=Q_1(\partial,\lambda)B,\ [A_\lambda B]=p(\lambda)B,\ [B_\lambda B]=0,
			\end{equation}
			for some polynomial $p(\lambda)$ and some skew-symmetric polynomial $Q_1(\partial,\lambda)$ satisfying \\$p(\lambda)Q_1(\partial,\lambda)=0$.
			
			\item[(2)] If $\mathcal{R}$ is neither solvable nor semisimple, then there are two classes.
			\begin{itemize}
				\item[(2i)] $\mathcal{R}$ is the direct sum of a rank one commutative Lie conformal algebra and the Virasoro Lie conformal algebra. That is, there is a basis $\{A,B\}$ of $\mathcal{R}$ satisfying 	\begin{equation}\label{a2}
				[A_\lambda A]=(\partial+2\lambda)A,\ [A_\lambda B]=0,\ [B_\lambda B]=0.
				\end{equation}
				\item[(2ii)] There is a basis $\{A,B\}$ of $\mathcal{R}$ such that
				\begin{equation}\label{a3}
				\begin{aligned}
					&[A_\lambda A]=(\partial+2\lambda)A+Q(\partial,\lambda)B,\\
			&[A_\lambda B]=(\partial+a \lambda +b)B,\ [B_\lambda B]=0,
				\end{aligned}
				\end{equation}
				where $a,b \in \mathbb{C}$ and   $Q(\partial,\lambda)$ is some skew-symmetric polynomial depending on $a,b$. Moreover, $Q(\partial,\lambda)\neq 0$ only when $a\in \{1, 0, -1, -4, -6 \}$ and $b=0$, in which case we document the explicit formula for $Q(\partial,\lambda)$ in the following table. 
				\begin{center}
				\begin{tabular}{|c|c|}
					\hline 
					$a$ &  $Q(\partial,\lambda)$, $c, d \in \C$,  \\ 
					\hline 
					$1$ &  $c(\partial+2\lambda)$  \\ 
					\hline 
					$0$ &  $c(\partial+2\lambda)(\partial+\lambda)\lambda+d(\partial+2\lambda)\partial$  \\ 
					\hline 
					$-1$ &  $c(\partial+2\lambda)\partial^2+d(\partial+2\lambda)(\partial+\lambda)\partial\lambda$  \\ 
					\hline 
					$-4$ &  $c(\partial+2\lambda)(\partial+\lambda)^3\lambda^3$  \\ 
					\hline
					$-6$ &  $c(\partial+2\lambda)[11(\partial+\lambda)^4\lambda^4+2(\partial+\lambda)^3\partial^2\lambda^3]$  \\ 
					\hline 
				\end{tabular} 
				\end{center}	
			\end{itemize}
		\end{itemize}
	\end{proposition}
\begin{remark}
	A polynomial $Q(\partial,\lambda)$ is called \textbf{skew-symmetric} if $Q(\partial,\lambda)=-Q(\partial,-\lambda-\partial)$.
\end{remark}
In this study, we refer to the two classes of non-solvable and non-semisimple rank two conformal algebras mentioned above as \textbf{Lie conformal algebras of Type I and Type II}, respectively. The complete classification of the finite nontrivial irreducible modules of a current conformal algebra, and its semidirect sum with the Virasoro conformal algebra, was provided in \cite{ck}, while that of $\mathcal{R}$ of other classes was described in \cite{xhw}.
\begin{proposition}[\cite{xhw}, Theorem 3.2]\label{m}  Suppose that $\mathcal{R}=\mathbb{C}[\partial]A\oplus \mathbb{C}[\partial]B$ is a Lie conformal algebra of rank two. Then any non-trivial finite irreducible $\mathcal{R}$-module is free of rank one. Moreover,  if  $V=\mathbb{C}[\partial]v$ is a non-trivial irreducible $\mathcal{R}$-module, then  the action of $\mathcal{R}$ on $V$ has to be  one of the  following cases:
	\begin{itemize}
		\item[(i)] If $\mathcal{R}=\mathbb{C}[\partial]A \oplus \mathbb{C}[\partial]B$ is a direct sum of two   Virasoro Lie conformal algebras with $[A_\lambda B]=0$,  then  either
		\[ A_\lambda v=(\partial+\alpha_1\lambda+\beta_1)v,\ B_\lambda v=0, \  \text{for}\ \text{some}  \   \beta_1,\  0 \neq\alpha_1 \in \mathbb{C},\]
		or
		\[ A_\lambda v=0,\ B_\lambda v=(\partial+\alpha_2\lambda+\beta_2)v, \    \text{for}\ \text{some}  \   \beta_2,\  0 \neq\alpha_2 \in \mathbb{C}.\]
		\item[(ii)] If $\mathcal{R}$ is solvable with the relations (\ref{a1}), then we have
		$ A_\lambda v=\phi_A(\lambda)v$, $B_\lambda v=\phi_B(\lambda)v$,
		where $\phi_A(\lambda)$, $\phi_B(\lambda)$ are not zero  simultaneously.
		Moreover, $\phi_B(\lambda) \neq 0 $ only if $p(\lambda)=Q_1(\partial,\lambda)=0$.
		\item[(iii)] Suppose that $\mathcal{R}$ is the Lie conformal algebra defined in (\ref{a2}), then either \[ A_\lambda v=(\partial+\alpha\lambda+\beta)v,\ B_\lambda v=0, \  \text{for}\ \text{some}  \   \beta,\  0 \neq\alpha \in \mathbb{C},\]
		or
		\[ A_\lambda v=0,\ B_\lambda v=\phi(\lambda)v, \  \text{for}\ \text{some}\ \text{nonzero}\  \phi(\lambda)\in \mathbb{C}[\lambda].\]
		\item[(iv)]  Suppose that $\mathcal{R}$ is the Lie conformal algebra defined in (\ref{a3}). Then
		\[ A_\lambda v=(\partial+\alpha\lambda+\beta),\ B_\lambda v=\gamma v,\]
		where $\alpha, \beta, \gamma \in \mathbb{C}$ such that $\gamma \neq 0$ only if $a=1$, $b=0$ and $Q(\partial,\lambda)=0$. Further, if $\gamma=0$, then $\alpha \neq 0$.
	\end{itemize}
\end{proposition}

	\section{For semisimple rank two Lie conformal algebras}\label{s3}
	
	In this section, we consider $\mathcal{R}$ as a semisimple rank two Lie conformal algebra. Then $\mathcal{R}$ is isomorphic to either the direct sum of two Virasoro Lie conformal algebras, or $Cur(\mathfrak{g})$ with $ dim(\mathfrak{g})=2 $, or $Vir\ltimes Cur(\mathfrak{h})$ with $dim(\mathfrak{h})=1$. Since the extension problem of $Cur(\mathfrak{g})$ and $Vir\ltimes Cur(\mathfrak{h})$ has been discussed in \cite{ckw1}, we can assume $\mathcal{R}=\mathbb{C}[\partial]A\oplus\mathbb{C}[\partial]B$ with
	\begin{equation}
	[A_\lambda A]=(\partial+2\lambda)A,\ [A_\lambda B]=0,\ [B_\lambda B]=(\partial+2\lambda)B.
	\end{equation}
	Let $V$ be a non-trivial finite irreducible $\mathcal{R}$-module. According to case (i) in Proposition \ref{m}, \begin{equation}
	V\cong V_{\delta,\alpha,\beta}=\mathbb{C}[\partial]v,\quad A_\lambda v=\delta_1(\partial+\alpha_1\lambda+\beta_1)v,\quad B_\lambda v=\delta_2(\partial+\alpha_2\lambda+\beta_2)v,
	\end{equation}
	where $\delta_i\in\{0,1\},\beta_i,0\neq\alpha_i\in\mathbb{C}$ for $i=1,2$, and $\delta_1^2+\delta_2^2=1$. 
	
	By Definition \ref{d2}, a $\mathcal{R}$-module structure on $V$ is given by $A_\lambda,B_\lambda\in End_\mathbb{C}(V)[\lambda]$ such that
	\begin{align}
	&[A_{\lambda},A_{\mu}]=(\lambda-\mu)A_{\lambda+\mu}, \label{m21}\\
	&[A_{\lambda},B_{\mu}]=0, \label{m22}\\
	&[B_{\lambda},B_{\mu}]=(\lambda-\mu)B_{\lambda+\mu}, \label{m23}\\
	&[\partial,A_{\lambda}]=-\lambda A_{\lambda}, \label{m24}\\
	&[\partial,B_{\lambda}]=-\lambda B_{\lambda}. \label{m25}
	\end{align}
	
	\subsection{$0\longrightarrow \mathbb{C}c_{\eta} \longrightarrow  E\longrightarrow  V_{\delta,\alpha,\beta} \longrightarrow   0$}\hfill \\

	First, we consider extensions of finite irreducible $\mathcal{R}$-modules of the form \begin{align}\label{f21}
	0\longrightarrow \mathbb{C}c_{\eta} \longrightarrow  E\longrightarrow  V_{\delta,\alpha,\beta} \longrightarrow   0.
	\end{align}Then $E$ is isomorphic to $\mathbb{C}c_{\eta}\oplus V_{\delta,\alpha,\beta}=\mathbb{C}c_{\eta}\oplus \mathbb{C}[\partial]v$ as a vector space, and the following identities hold in $E$:
	\begin{align}\label{se4}
	&\mathcal{R}_\lambda c_{\eta}=0,\quad\partial c_{\eta}=\eta c_{\eta}, \notag\\
	&A_\lambda v=\delta_1(\partial+\alpha_1\lambda+\beta_1)v+f(\lambda)c_{\eta},\quad B_\lambda v=\delta_2(\partial+\alpha_2\lambda+\beta_2)v+g(\lambda)c_{\eta}, 
	\end{align}
	where $f(\lambda), g(\lambda) \in \mathbb{C}[\lambda]$.
	
	\begin{lemma}\label{l21}
		All trivial extensions of finite irreducible $\mathcal{R}$-modules of the form (\ref{f21}) are given by (\ref{se4}), where $f(\lambda)$ and $g(\lambda)$ are the same scalar multiples of $\delta_1(\alpha_1\lambda+\eta+\beta_1)$ and $\delta_2(\alpha_2\lambda+\eta+\beta_2)$, respectively.
	\end{lemma}
	\begin{proof}
		Assume that (\ref{f21}) is a trivial extension, that is, there exists $v'=kc_\eta+l(\partial)v\in E$, where $k\in\mathbb{C}$ and $0\neq l(\partial)\in\mathbb{C}[\partial]$, such that \begin{align*}
		&A_\lambda v'=\delta_1(\partial+\alpha_1\lambda+\beta_1)v'=k\delta_1(\eta+\alpha_1\lambda+\beta_1)c_\eta+\delta_1l(\partial)(\partial+\alpha_1\lambda+\beta_1)v,\\
		&B_\lambda v'=\delta_2(\partial+\alpha_2\lambda+\beta_2)v'=k\delta_2(\eta+\alpha_2\lambda+\beta_2)c_\eta+\delta_2l(\partial)(\partial+\alpha_2\lambda+\beta_2)v.
		\end{align*}
		
		On the other hand, it follows from (\ref{se4}) that
		\begin{align*}
		&A_\lambda v'=f(\lambda)l(\eta+\lambda)c_\eta+\delta_1l(\partial+\lambda)(\partial+\alpha_1\lambda+\beta_1)v,\\
		&B_\lambda v'=g(\lambda)l(\eta+\lambda)c_\eta+\delta_2l(\partial+\lambda)(\partial+\alpha_2\lambda+\beta_2)v.
		\end{align*}
		We can obtain that $l(\partial)$ is a nonzero constant by comparing both expressions for $A_\lambda v'$ and $B_\lambda v'$. Thus $f(\lambda)$ and $g(\lambda)$ are the same scalar multiple of $\delta_1(\alpha_1\lambda+\eta+\beta_1)$ and $\delta_2(\alpha_2\lambda+\eta+\beta_2)$, respectively.
		
		Conversely, if $f(\lambda)=k\delta_1(\alpha_1\lambda+\eta+\beta_1)$ and $g(\lambda)=k\delta_2(\alpha_2\lambda+\eta+\beta_2)$ for some $k\in\mathbb{C}$, setting $v'=kc_\eta +v$ we can deduce that (\ref{f21}) is a trivial extension.
	\end{proof}

\begin{theorem}\label{t21}
	Let $\mathcal{R}$ be a direct sum of two Virasoro conformal algebras. Then nontrivial extensions of finite irreducible conformal modules of the form (\ref{f21}) exist only when $(\delta_1,\delta_2)=(1,0),\alpha_1\in\{1,2\},\beta_1+\eta=0$ or $(\delta_1,\delta_2)=(0,1),\alpha_2\in\{1,2\},\beta_2+\eta=0$. In these cases, there exists a unique (up to a scalar) nontrivial extension, i.e. $dim(V_{\delta,\alpha,\beta},Ext(\mathbb{C}c_\eta))=1$. Moreover, they are given (up to equivalence) by (\ref{se4}). The values of $\delta_i,\alpha_i,\beta_i,i=1,2$ and $\eta$, along with the corresponding polynomials $f(\lambda)$ and $g(\lambda)$ giving rise to nontrivial extensions, are listed as follows: \begin{enumerate}[(i)]
		\item If $\delta_1=1,\alpha_1\in\{1,2\},\beta_1+\eta=0$, then $\delta_2=0,g(\lambda)=0,0\neq\alpha_2,\beta_2\in\mathbb{C}$ and \begin{align*}
		f(\lambda)=\begin{cases}
		s_1\lambda^2,&\alpha_1=1,\\
		s_2\lambda^3,&\alpha_1=2,
		\end{cases}
		\end{align*}
		with nonzero constants $s_1,s_2$.
		\item If $\delta_2=1,\alpha_2\in\{1,2\},\beta_2+\eta=0$, then $\delta_1=0,f(\lambda)=0,0\neq\alpha_1,\beta_1\in\mathbb{C}$ and \begin{align*}
		g(\lambda)=\begin{cases}
		t_1\lambda^2,&\alpha_2=1,\\
		t_2\lambda^3,&\alpha_2=2,
		\end{cases}
		\end{align*}
		with nonzero constants $t_1,t_2$.
	\end{enumerate}

\end{theorem}
\begin{proof}
	Applying both sides of (\ref{m22}) to $v$ and comparing the corresponding efficients, we obtain \begin{align}
	&\delta_2(\eta+\lambda+\alpha_2\mu+\beta_2)f(\lambda)-\delta_1(\eta+\mu+\alpha_1\lambda+\beta_1)g(\mu)=0,\label{te42}
	\end{align}
	
	If $(\delta_1,\delta_2)=(1,0)$, (\ref{te42}) implies $g(\mu)=0$ and it reduces to the case of Virasoro conformal algebra. We can deduce the result by Proposition 2.1 in \cite{ckw1}. A similar discussion can be made with $(\delta_1,\delta_2)=(0,1)$. 
\end{proof}

\subsection{$0\longrightarrow V_{\delta,\alpha,\beta} \longrightarrow  E\longrightarrow\mathbb{C}c_{\eta}   \longrightarrow   0$}\hfill \\

Next, we consider extensions of finite irreducible $\mathcal{R}$-modules of the form \begin{align}\label{f22}
0\longrightarrow V_{\delta,\alpha,\beta} \longrightarrow  E\longrightarrow\mathbb{C}c_{\eta}   \longrightarrow   0.
\end{align}Then $E$ is isomorphic to $V_{\delta,\alpha,\beta}\oplus\mathbb{C}c_{\eta} =\mathbb{C}[\partial]v\oplus\mathbb{C}c_{\eta} $ as a vector space, and the following identities hold in $E$:
\begin{align}\label{se5}
&A_\lambda v=\delta_1(\partial+\alpha_1\lambda+\beta_1)v,\quad B_\lambda v=\delta_2(\partial+\alpha_2\lambda+\beta_2)v, \notag\\
&A_\lambda c_{\eta}=f(\partial,\lambda)v,\quad B_\lambda c_{\eta}=g(\partial,\lambda)v,\quad\partial c_{\eta}=\eta c_{\eta}+h(\partial)v, 
\end{align}
where $f(\partial,\lambda), g(\partial,\lambda), h(\partial)\in \mathbb{C}[\partial,\lambda]$.

\begin{lemma}\label{l22}
	All trivial extensions of finite irreducible $\mathcal{R}$-modules of the form (\ref{f22}) are given by (\ref{se5}), and $f(\partial,\lambda)=\delta_1\varphi(\partial+\lambda)(\partial+\alpha_1\lambda+\beta_1),g(\partial,\lambda)=\delta_2\varphi(\partial+\lambda)(\partial+\alpha_2\lambda+\beta_2)$ and $h(\partial)=(\partial-\eta)\varphi(\partial)$, where $\varphi$ is a polynomial.
\end{lemma}
\begin{proof}
	Assume that (\ref{f22}) is a trivial extension, that is, there exists $c'_\eta=kc_\eta+l(\partial)v\in E$, where $0\neq k\in\mathbb{C}$ and $ l(\partial)\in\mathbb{C}[\partial]$, such that $A_\lambda c'_\eta=B_\lambda c'_\eta=0$ and $\partial c'_{\eta}=\eta c'_{\eta}$.
	
	On the other hand, it follows from (\ref{se5}) that
	\begin{align*}
	&A_\lambda c'_\eta=(kf(\partial,\lambda)+\delta_1l(\partial+\lambda)(\partial+\alpha_1\lambda+\beta_1))v,\\
	&B_\lambda c'_\eta=(kg(\partial,\lambda)+\delta_2l(\partial+\lambda)(\partial+\alpha_2\lambda+\beta_2))v,\\
	&\partial c'_{\eta}=k\eta c_\eta+(kh(\partial)+\partial l(\partial))v.
	\end{align*}
	We can obtain the result by comparing both expressions for  $A_\lambda c'_\eta,B_\lambda c'_\eta$ and $\partial c'_\eta$.
	
	Conversely, if $f(\partial,\lambda)=\delta_1\varphi(\partial+\lambda)(\partial+\alpha_1\lambda+\beta_1),g(\partial,\lambda)=\delta_2\varphi(\partial+\lambda)(\partial+\alpha_2\lambda+\beta_2)$ and $h(\partial)=(\partial-\eta)\varphi(\partial)$ for some polynomial $\varphi$, setting $c'_{\eta}=c_\eta -\varphi(\partial)v$, we can deduce that (\ref{f22}) is a trivial extension.
\end{proof}

\begin{theorem}\label{t22}
	Let $\mathcal{R}$ be a direct sum of two Virasoro conformal algebras. Then nontrivial extensions of finite irreducible conformal modules of the form (\ref{f22}) exist only when $\delta_1=1,\alpha_1=1,\beta_1+\eta=0$ or $\delta_2=1,\alpha_2=1,\beta_2+\eta=0$. In these cases, there exists a unique (up to a scalar) nontrivial extension, i.e. $dim(Ext(\mathbb{C}c_\eta,V_{\delta,\alpha,\beta}))=1$. Moreover, they are given (up to equivalence) by (\ref{se5}). The values of $\delta_i,\alpha_i,\beta_i,i=1,2$ and $\eta$ along with the corresponding polynomials $f(\partial,\lambda),g(\partial,\lambda)$ and $h(\partial)$ giving rise to nontrivial extensions, are listed as follows: \begin{enumerate}[(i)]
		\item If $\delta_1=1,\alpha_1=1,\beta_1+\eta=0$, then $\delta_2=0,g(\partial,\lambda)=0,\alpha_2,\beta_2\in\mathbb{C}$ and $f(\partial,\lambda)=h(\partial)=s$ with nonzero constant $s$.
		\item If $\delta_2=1,\alpha_2=1,\beta_2+\eta=0$, then $\delta_1=0,f(\partial,\lambda)=0,\alpha_1,\beta_1\in\mathbb{C}$ and $g(\partial,\lambda)=h(\partial)=t$ with nonzero constant $t$.
	\end{enumerate}
\end{theorem}
\begin{proof}
	Applying both sides of (\ref{m24}) and (\ref{m25})  to $c_\eta$ and comparing the corresponding efficients gives the following equations \begin{align}
	&(\partial+\lambda-\eta)f(\partial,\lambda)=\delta_1h(\partial+\lambda)(\partial+\alpha_1\lambda+\beta_1),\label{te51}\\
	&(\partial+\lambda-\eta)g(\partial,\lambda)=\delta_2h(\partial+\lambda)(\partial+\alpha_2\lambda+\beta_2).\label{te52}
	\end{align}

We only need to consider the case that $(\delta_1,\delta_2)=(1,0)$. Then $g(\partial,\lambda)=0$ by (\ref{te52}) and the result can be deduced by Proposition 2.2 in \cite{ckw1}.
\end{proof}
\subsection{$0\longrightarrow V_{\delta,\alpha,\beta} \longrightarrow  E\longrightarrow  V_{\bar{\delta},\bar{\alpha},\bar{\beta}} \longrightarrow   0$}\hfill\\

Finally, we consider extensions of finite irreducible $\mathcal{R}$-modules of the form \begin{align}\label{f23}
0\longrightarrow V_{\delta,\alpha,\beta} \longrightarrow  E\longrightarrow  V_{\bar{\delta},\bar{\alpha},\bar{\beta}} \longrightarrow   0.
\end{align}Then $E$ is isomorphic to $V_{\delta,\alpha,\beta}\oplus V_{\bar{\delta},\bar{\alpha},\bar{\beta}} =\mathbb{C}[\partial]v\oplus\mathbb{C}[\partial]\bar{v} $ as a vector space, and the following identities hold in $E$:
\begin{align}\label{se6}
&A_\lambda v=\delta_1(\partial+\alpha_1\lambda+\beta_1)v,\quad B_\lambda v=\delta_2(\partial+\alpha_2\lambda+\beta_2)v, \notag\\
&A_\lambda \bar{v}=\bar{\delta}_1(\partial+\bar{\alpha}_1\lambda+\bar{\beta}_1)\bar{v}+f(\partial,\lambda)v,\quad B_\lambda \bar{v}=\bar{\delta}_2(\partial+\bar{\alpha}_2\lambda+\bar{\beta}_2)\bar{v}+g(\partial,\lambda)v, 
\end{align}
where $f(\partial,\lambda), g(\partial,\lambda)\in \mathbb{C}[\partial,\lambda]$.
\begin{lemma}\label{l23}
	All trivial extensions of finite irreducible $\mathcal{R}$-modules of the form (\ref{f23}) are given by (\ref{se6}), and $f(\partial,\lambda)=\delta_1\varphi(\partial+\lambda)(\partial+\alpha_1\lambda+\beta_1)-\bar{\delta}_1\varphi(\partial)(\partial+\bar{\alpha}_1\lambda+\bar{\beta}_1)$ and $g(\partial,\lambda)=\delta_2\varphi(\partial+\lambda)(\partial+\alpha_2\lambda+\beta_2)-\bar{\delta}_2\varphi(\partial)(\partial+\bar{\alpha}_2\lambda+\bar{\beta}_2)$ for some polynomial $\varphi$.
\end{lemma}
\begin{proof}
	Assume that (\ref{f23}) is a trivial extension, that is, there exists $\bar{v}'=k(\partial)v+l(\partial)\bar{v}\in E$, where $k(\partial),l(\partial)\in\mathbb{C}$ and $ l(\partial)\neq0$, such that \begin{align*}
	&A_\lambda \bar{v}'=\bar{\delta}_1(\partial+\bar{\alpha}_1\lambda+\bar{\beta}_1)\bar{v}'=\bar{\delta}_1k(\partial)(\partial+\bar{\alpha}_1\lambda+\bar{\beta}_1)v+\bar{\delta}_1l(\partial)(\partial+\bar{\alpha}_1\lambda+\bar{\beta}_1)\bar{v},\\
	&B_\lambda \bar{v}'=\bar{\delta}_2(\partial+\bar{\alpha}_2\lambda+\bar{\beta}_2)\bar{v}'=\bar{\delta}_2k(\partial)(\partial+\bar{\alpha}_2\lambda+\bar{\beta}_2)v+\bar{\delta}_2l(\partial)(\partial+\bar{\alpha}_2\lambda+\bar{\beta}_2)\bar{v}.
	\end{align*}
	
	On the other hand, it follows from (\ref{se6}) that
	\begin{align*}
	&A_\lambda \bar{v}'=(\delta_1k(\partial+\lambda)(\partial+\alpha_1\lambda+\beta_1)+l(\partial+\lambda)f(\partial,\lambda))v+\bar{\delta}_1l(\partial+\lambda)(\partial+\bar{\alpha}_1\lambda+\bar{\beta}_1)\bar{v},\\
	&B_\lambda \bar{v}'=(\delta_2k(\partial+\lambda)(\partial+\alpha_2\lambda+\beta_2)+l(\partial+\lambda)g(\partial,\lambda))v+\bar{\delta}_2l(\partial+\lambda)(\partial+\bar{\alpha}_2\lambda+\bar{\beta}_2)\bar{v}.
	\end{align*}
	Comparing both expressions for $A_\lambda \bar{v}'$ and $B_\lambda \bar{v}'$, we can obtain that $l(\partial)$ is a nonzero constant. Then we can give the expressions of $f(\partial,\lambda)$ and $g(\partial,\lambda)$.
	
	Conversely, if $f(\partial,\lambda)=\delta_1\varphi(\partial+\lambda)(\partial+\alpha_1\lambda+\beta_1)-\bar{\delta}_1\varphi(\partial)(\partial+\bar{\alpha}_1\lambda+\bar{\beta}_1)$ and $g(\partial,\lambda)=\delta_2\varphi(\partial+\lambda)(\partial+\alpha_2\lambda+\beta_2)-\bar{\delta}_2\varphi(\partial)(\partial+\bar{\alpha}_2\lambda+\bar{\beta}_2)$ for some polynomial $\varphi$, setting $\bar{v}'=-\varphi(\partial)v+\bar{v}$ we can deduce that (\ref{f23}) is a trivial extension.
\end{proof}

Before classifying all nontrivial extensions of the form (\ref{f23}), we give the following lemma for later use. 

\begin{lemma}\label{r3}
		The equation
	\begin{align}\label{r3e}
	c(\partial+\lambda,\mu)(\partial+a\lambda+b)-c(\partial,\mu)(\partial+\mu+\bar{a}\lambda+\bar{b})=0
	\end{align} for unknown polynomials $c(\partial,\lambda)\in\mathbb{C}[\partial,\lambda]$ with $a,\bar{a},b,\bar{b}\in\mathbb{C}$ has only zero solution.
\end{lemma}
\begin{proof}
  Putting $\lambda=0$ in (\ref{r3e}), we get $c(\partial,\mu)(b-\mu-\bar{b})=0$. So $c=0$.
\end{proof}
\begin{theorem}\label{t23}
	Let $\mathcal{R}$ be a direct sum of two Virasoro conformal algebras. Then nontrivial extensions of finite irreducible conformal modules of the form (\ref{f23}) only exist in the following cases. Moreover, they are given (up to equivalence) by (\ref{se6}). The value of $\delta_i,\bar{\delta}_i,\alpha_i,\bar{\alpha}_i,\beta_i,\bar{\beta}_i,i=1,2$ and the corresponding polynomials $f(\partial,\lambda)$ and $g(\partial,\lambda)$ giving rise to nontrivial extensions, are listed as follows:
	\begin{enumerate}
		\item In the case that $\delta_1=\bar{\delta}_1=1,\delta_2=\bar{\delta}_2=0$, $\alpha_2,\beta_2,\bar{\alpha}_2,\bar{\beta}_2\in\mathbb{C},g(\partial,\lambda)=0,\beta_1=\bar{\beta}_1,\bar{\alpha}_1-\alpha_1\in\{0,1,2,3,4,5,6\},\alpha_1,\bar{\alpha}_1\neq0$ and	\begin{enumerate}[(i)]
			\item $\bar{\alpha}_1=\alpha_1$, $f(\partial,\lambda)=s_0+s_1\lambda$, where $(s_0,s_1)\neq (0,0)$.
			\item $\bar{\alpha}_1-\alpha_1=2$, $f(\partial,\lambda)=s\lambda^2(2(\partial+\beta_1)+\lambda)$, where $s\neq0$.
			\item $\bar{\alpha}_1-\alpha_1=3$, $f(\partial,\lambda)=s(\partial+\beta_1)\lambda^2(\partial+\beta_1+\lambda)$, where $s\neq0$.
			\item $\bar{\alpha}_1-\alpha_1=4$, $f(\partial,\lambda)=s\lambda^2(4(\partial+\beta_1)^3+6(\partial+\beta_1)^2\lambda-(\partial+\beta_1)\lambda^2+\alpha_1\lambda^3)$, where $s\neq0$.
			\item $\bar{\alpha}_1=1$ and $\alpha_1=-4$, $f(\partial,\lambda)=s((\partial+\beta_1)^4\lambda^2-10(\partial+\beta_1)^2\lambda^4-17(\partial+\beta_1)\lambda^5-8\lambda^6)$, where $s\neq0$.
			
			\item $\bar{\alpha}_1=\frac{7}{2}\pm\frac{\sqrt{19}}{2}$ and $\alpha_1=-\frac{5}{2}\pm\frac{\sqrt{19}}{2}$. $f(\partial,\lambda)=s((\partial+\beta_1)^4\lambda^3-(2\alpha_1+3)(\partial+\beta_1)^3\lambda^4-3\alpha_1(\partial+\beta_1)^2\lambda^5-(3\alpha_1+1)(\partial+\beta_1)\lambda^6-(\alpha_1+\frac{9}{28})\lambda^7)$, where $s\neq0$.
		\end{enumerate}
	The value of $dim(Ext(V_{\bar{\delta},\bar{\alpha},\bar{\beta}},V_{\delta,\alpha,\beta}))$ is 2 in subcase (i), and 1 in subcases (ii)-(vi).
	
	\item In the case that $\delta_1=\bar{\delta}_1=0,\delta_2=\bar{\delta}_2=1$,  $\alpha_1,\beta_1,\bar{\alpha}_1,\bar{\beta}_1\in\mathbb{C},f(\partial,\lambda)=0,\beta_2=\bar{\beta}_2,\bar{\alpha}_2-\alpha_2\in\{0,1,2,3,4,5,6\},\alpha_2,\bar{\alpha}_2\neq0,$ and	\begin{enumerate}[(i)]
		\item $\bar{\alpha}_2=\alpha_2$, $g(\partial,\lambda)=t_0+t_1\lambda$, where $(t_0,t_1)\neq (0,0)$.
		\item $\bar{\alpha}_2-\alpha_2=2$, $g(\partial,\lambda)=t\lambda^2(2(\partial+\beta_2)+\lambda)$, where $t\neq0$.
		\item $\bar{\alpha}_2-\alpha_2=3$, $g(\partial,\lambda)=t(\partial+\beta_2)\lambda^2(\partial+\beta_2+\lambda)$, where $t\neq0$.
		\item $\bar{\alpha}_2-\alpha_2=4$, $g(\partial,\lambda)=t\lambda^2(4(\partial+\beta_2)^3+6(\partial+\beta_2)^2\lambda-(\partial+\beta_2)\lambda^2+\alpha_2\lambda^3)$, where $t\neq0$.
		\item $\bar{\alpha}_2=1$ and $\alpha_2=-4$, $g(\partial,\lambda)=t((\partial+\beta_2)^4\lambda^2-10(\partial+\beta_2)^2\lambda^4-17(\partial+\beta_2)\lambda^5-8\lambda^6)$, where $t\neq0$.
		
		\item $\bar{\alpha}_2=\frac{7}{2}\pm\frac{\sqrt{19}}{2}$ and $\alpha_2=-\frac{5}{2}\pm\frac{\sqrt{19}}{2}$, $g(\partial,\lambda)=t((\partial+\beta_2)^4\lambda^3-(2\alpha_2+3)(\partial+\beta_2)^3\lambda^4-3\alpha_2(\partial+\beta_2)^2\lambda^5-(3\alpha_2+1)(\partial+\beta_2)\lambda^6-(\alpha_2+\frac{9}{28})\lambda^7)$, where $t\neq0$.
	\end{enumerate}
The value of $dim(Ext(V_{\bar{\delta},\bar{\alpha},\bar{\beta}},V_{\delta,\alpha,\beta}))$ is 2 in subcase (i), and 1 in subcases (ii)-(vi).	
	\end{enumerate}
\end{theorem}
\begin{proof}
	Applying both sides of (\ref{m21}), (\ref{m22}) and (\ref{m23}) to $\bar{v}$ and comparing the corresponding coefficients of $v$, we obtain \begin{align}
	&\bar{\delta}_1f(\partial,\lambda)(\partial+\lambda+\bar{\alpha}_1\mu+\bar{\beta}_1)+\delta_1f(\partial+\lambda,\mu)(\partial+\alpha_1\lambda+\beta_1)\notag\\
	&\quad-\bar{\delta}_1f(\partial,\mu)(\partial+\mu+\bar{\alpha}_1\lambda+\bar{\beta}_1)-\delta_1f(\partial+\mu,\lambda)(\partial+\alpha_1\mu+\beta_1)=(\lambda-\mu)f(\partial,\lambda+\mu),\label{te61}\\
	&\bar{\delta}_2f(\partial,\lambda)(\partial+\lambda+\bar{\alpha}_2\mu+\bar{\beta}_2)+\delta_1g(\partial+\lambda,\mu)(\partial+\alpha_1\lambda+\beta_1)\notag\\
	&\quad-\bar{\delta}_1g(\partial,\mu)(\partial+\mu+\bar{\alpha}_1\lambda+\bar{\beta}_1)-\delta_2f(\partial+\mu,\lambda)(\partial+\alpha_2\mu+\beta_2)=0,\label{te62}\\
	&\bar{\delta}_2g(\partial,\lambda)(\partial+\lambda+\bar{\alpha}_2\mu+\bar{\beta}_2)+\delta_2g(\partial+\lambda,\mu)(\partial+\alpha_2\lambda+\beta_2)\notag\\
	&\quad-\bar{\delta}_2g(\partial,\mu)(\partial+\mu+\bar{\alpha}_2\lambda+\bar{\beta}_2)-\delta_2g(\partial+\mu,\lambda)(\partial+\alpha_2\mu+\beta_2)=(\lambda-\mu)g(\partial,\lambda+\mu).\label{te63}
	\end{align}
	
	If $(\delta_1,\bar{\delta}_1,\delta_2,\bar{\delta}_2)=(1,1,0,0)$ or $(\delta_1,\bar{\delta}_1,\delta_2,\bar{\delta}_2)=(0,0,1,1)$, $g(\partial,\lambda)=0$ or $f(\partial,\lambda)=0$ follows from (\ref{te62}) and Lemma \ref{r3}, and then the result can be inferred from Theorem 3.2 in \cite{ckw1} (or Theorem 2.7 in \cite{lhw2}).
	
	If $(\delta_1,\bar{\delta}_1,\delta_2,\bar{\delta}_2)=(1,0,0,1)$, then putting $\mu=0$ in (\ref{te61}), we can obtain $$f(\partial+\lambda,0)(\partial+\alpha_1\lambda+\beta_1)=f(\partial,\lambda)(\partial+\lambda+\beta_1).$$ So when $\alpha_1=1$, we have $f(\partial,\lambda)=f(\partial+\lambda,0)=s(\partial+\lambda)$ for some polynomial $s$. If $\alpha_1\neq1$, we can denote $f(\partial,\lambda)=h(\partial,\lambda)(\partial+\alpha_1\lambda+\beta_1)$, and thus $f(\partial+\lambda,0)=h(\partial+\lambda,0)(\partial+\lambda+\beta_1)$. Then one can deduce that $h(\partial,\lambda)=h(\partial+\lambda,0)$. It is not difficult to check that $f(\partial,\lambda)=s(\partial+\lambda)(\partial+\alpha_1\lambda+\beta_1)$ for some polynomial $s$. On the other hand, dealing with (\ref{te63}) in a similar way, we have $
	g(\partial,\lambda)=
	t(\partial)(\partial+\bar{\alpha}_2\lambda+\bar{\beta}_2)$ for $\bar{\alpha}_2\neq0$, where $t$ is a polynomial. Putting these results in (\ref{te62}), we can obtain\begin{align}
	\begin{cases}
	s(\partial+\lambda)(\partial+\lambda+\bar{\alpha}_2\mu+\bar{\beta}_2)+t(\partial+\lambda)(\partial+\lambda+\bar{\alpha}_2\mu+\bar{\beta}_2)(\partial+\lambda+\beta_1)=0,&\alpha_1=1,\\
	s(\partial+\lambda)(\partial+\alpha_1\lambda+\beta_1)(\partial+\lambda+\bar{\alpha}_2\mu+\bar{\beta}_2)\\
	\qquad+t(\partial+\lambda)(\partial+\lambda+\bar{\alpha}_2\mu+\bar{\beta}_2)(\partial+\alpha_1\lambda+\beta_1)=0,&\alpha_1\neq1.\\
	\end{cases}	
	\end{align}
	The solutions are concluded as follows.
	 \begin{enumerate}[(i)]
		\item If $\alpha_1=1$, then $f(\partial,\lambda)=-t(\partial+\lambda)(\partial+\lambda+\beta_1),g(\partial,\lambda)=t(\partial)(\partial+\bar{\alpha}_2\lambda+\bar{\beta}_2)$ for some polynomial $t$. The extension is trivial.
		\item If $\alpha_1\neq1$, then $f(\partial,\lambda)=s(\partial+\lambda)(\partial+\alpha_1\lambda+\beta_1),g(\partial,\lambda)=-s(\partial)(\partial+\bar{\alpha}_2\lambda+\bar{\beta}_2)$ for some polynomial $t$. The extension is trivial.
	\end{enumerate}
	
	If $(\delta_1,\bar{\delta}_1,\delta_2,\bar{\delta}_2)=(0,1,1,0)$, one can deduce the result similarly.
	
\end{proof}
\section{For solvable rank two Lie conformal algebras }\label{s4}
In this section, we classify the extension of irreducible modules over a solvable rank two  Lie conformal algebra $\mathcal{R}$. Then there is a basis $\{A,B\}$ of $\mathcal{R}$ such that
\begin{equation}
[A_\lambda A]=Q_1(\partial,\lambda)B,\ [A_\lambda B]=p(\lambda)B,\ [B_\lambda B]=0,
\end{equation}
for some polynomial $p(\lambda)$ and some skew-symmetric polynomial $Q_1(\partial,\lambda)$ satisfying $p(\lambda)Q_1(\partial,\lambda)=0$ \cite{bch}. 
If $V$ is a non-trivial finite irreducible $\mathcal{R}$-module, it was shown in \cite{xhw} that \begin{equation}
V\cong V_{\phi_A,\phi_B}=\mathbb{C}[\partial]v,\quad A_\lambda v=\phi_A(\lambda)v,\quad B_\lambda v=\phi_B(\lambda)v,
\end{equation}
where $\phi_A(\lambda)$, $\phi_B(\lambda)$ are not zero  simultaneously.
Moreover, $\phi_B(\lambda) \neq 0 $ only if $p(\lambda)=Q_1(\partial,\lambda)=0$.

By definition \ref{d2}, a $\mathcal{R}$-module structure on $V$ is given by $A_\lambda,B_\lambda\in End_\mathbb{C}(V)[\lambda]$ such that
\begin{align}
&[A_{\lambda},A_{\mu}]=Q_1(-\lambda-\mu,\lambda)B_{\lambda+\mu}, \label{m11}\\
&[A_{\lambda},B_{\mu}]=p(\lambda)B_{\lambda+\mu}, \label{m12}\\
&[B_{\lambda},B_{\mu}]=0, \label{m13}\\
&[\partial,A_{\lambda}]=-\lambda A_{\lambda}, \label{m14}\\
&[\partial,B_{\lambda}]=-\lambda B_{\lambda}. \label{m15}
\end{align}
\subsection{$0\longrightarrow \mathbb{C}c_{\eta} \longrightarrow  E\longrightarrow  V_{\phi_A,\phi_B} \longrightarrow   0$}\hfill\\

First, we consider extensions of finite irreducible $\mathcal{R}$-modules of the form \begin{align}\label{f11}
0\longrightarrow \mathbb{C}c_{\eta} \longrightarrow  E\longrightarrow  V_{\phi_A,\phi_B} \longrightarrow   0.
\end{align}Then $E$ is isomorphic to $\mathbb{C}c_{\eta}\oplus V_{\phi_A,\phi_B}=\mathbb{C}c_{\eta}\oplus \mathbb{C}[\partial]v$ as a vector space, and the following identities hold in $E$:
\begin{align}\label{se1}
&\mathcal{R}_\lambda c_{\eta}=0,\quad\partial c_{\eta}=\eta c_{\eta}, \notag\\
&A_\lambda v=\phi_A(\lambda)v+f(\lambda)c_{\eta},\quad B_\lambda v=\phi_B(\lambda)v+g(\lambda)c_{\eta}, 
\end{align}
where $f(\lambda), g(\lambda) \in \mathbb{C}[\lambda]$.

\begin{lemma}\label{l11}
	All trivial extensions of finite irreducible $\mathcal{R}$-modules of the form (\ref{f11}) are given by (\ref{se1}), and $f(\lambda),g(\lambda)$ are the same scalar multiples of $\phi_A(\lambda),\phi_B(\lambda)$, respectively.
\end{lemma}
\begin{proof}
	Assume that (\ref{f11}) is a trivial extension, that is, there exists $v'=kc_\eta+l(\partial)v\in E$, where $k\in\mathbb{C}$ and $0\neq l(\partial)\in\mathbb{C}[\partial]$, such that \begin{align*}
	&A_\lambda v'=\phi_A(\lambda)v'=k\phi_A(\lambda)c_\eta+\phi_A(\lambda)l(\partial)v,\\
	&B_\lambda v'=\phi_B(\lambda)v'=k\phi_B(\lambda)c_\eta+\phi_B(\lambda)l(\partial)v.
	\end{align*}
	
	On the other hand, it follows from (\ref{se1}) that
	\begin{align*}
	&A_\lambda v'=f(\lambda)l(\eta+\lambda)c_\eta+\phi_A(\lambda)l(\partial+\lambda)v,\\
	&B_\lambda v'=g(\lambda)l(\eta+\lambda)c_\eta+\phi_B(\lambda)l(\partial+\lambda)v.
	\end{align*}
	We can obtain that $l(\partial)$ is a nonzero constant by comparing both expressions for $A_\lambda v'$ and $B_\lambda v'$. Thus $f(\lambda)$ and $g(\lambda)$ are the same scalar multiple of $\phi_A(\lambda)$ and $\phi_B(\lambda)$ respectively.
	
	Conversely, if $f(\lambda)=k\phi_A(\lambda)$ and $g(\lambda)=k\phi_B(\lambda)$ for some $k\in\mathbb{C}$, setting $v'=kc_\eta +v$ we can deduce that (\ref{f11}) is a trivial extension.
\end{proof}

The following key lemma plays a crucial role
in simplifying the calculations in our classification of nontrivial extensions.
\begin{lemma}\label{r1}
	Let $a(\lambda),b(\mu),c(\lambda),d(\mu)$ be four polynomials in $\mathbb{C}[\lambda,\mu]$. If $a(\lambda)$ and $b(\mu)$ are not 0 simultaneously, the equation $$a(\lambda)d(\mu)-b(\mu)c(\lambda)=0$$implies that $c(\lambda)$ and $d(\mu)$ are the same multiples of $a(\lambda)$ and $b(\mu)$ respectively. Particularly, if both $a(\lambda)$ and $b(\mu)$ are not 0, the multiple is a scalar multiple.
\end{lemma}
\begin{proof}
	Without loss of generality, we assume that $a(\lambda)\neq 0$. If $b(\mu)=0$, then $d(\mu)=0$. In this case, the conclusion is founded whatever $c(\lambda)$ is. If $b(\mu)\neq0$, we have $\frac{c(\lambda)}{a(\lambda)}=\frac{d(\mu)}{b(\mu)}:=e(\lambda,\mu)$. It is easy to see that $e(\lambda,\mu)$ is a constant, and the proof is done. 
\end{proof}
\begin{theorem}\label{t11}
	For solvable rank two Lie conformal algebra $\mathcal{R}$, nontrivial extensions of finite irreducible conformal modules of the form (\ref{f11}) exist only if $p(\lambda)\neq 0$ and $Q_1(\partial,\lambda)=0$. Moreover, they are given (up to equivalence) by (\ref{se1}). The values of $\eta$ along with the corresponding polynomials $\phi_A(\lambda),\phi_B(\lambda),f(\lambda)$ and $g(\lambda)$ giving rise to nontrivial extensions, are listed as follows: $\eta\in\mathbb{C}$,
	$\phi_A(\lambda)=-p(\lambda)$, $\phi_B(\lambda)=0$, $f(\lambda)=0$ and $g(\lambda)$ is a nonzero constant. Thus $dim(Ext(V_{\phi_A,\phi_B},\mathbb{C}c_\eta))=\delta_{\phi_A(\lambda)+p(\lambda),0}$.

\end{theorem}
\begin{proof}
	Applying both sides of (\ref{m11}), (\ref{m12}) and (\ref{m13}) to $v$ and comparing the corresponding efficients, we obtain \begin{align}
	&Q_1(-\lambda-\mu,\lambda)\phi_B(\lambda+\mu)=p(\lambda)\phi_B(\lambda+\mu)=0,\label{te11}\\
	&f(\lambda)\phi_A(\mu)-f(\mu)\phi_A(\lambda)=Q_1(-\lambda-\mu,\lambda)g(\lambda+\mu),\label{te12}\\
	&f(\lambda)\phi_B(\mu)-g(\mu)\phi_A(\lambda)=p(\lambda)g(\lambda+\mu),\label{te13}\\
	&g(\lambda)\phi_B(\mu)-g(\mu)\phi_B(\lambda)=0.\label{te14}
	\end{align}
	
	We first consider the case that $p(\lambda)=Q_1(\partial,\lambda)=0$ and take it in (\ref{te12}), (\ref{te13}) and (\ref{te14}). Since $\phi_A(\lambda)$, $\phi_B(\lambda)$ are not zero  simultaneously, we can assume that $\phi_A(\lambda)\neq 0$. By Lemma \ref{r1} and (\ref{te12}), we have $f(\lambda)=k\phi_A(\lambda)$ for some $k\in\mathbb{C}$. It can be deduced from (\ref{te13}) that $g(\mu)=k\phi_B(\mu)$ for the same $k$. According to Lemma \ref{l11}, the extension is trivial in this case. 
	
	Now we assume that $p(\lambda)=0$ and $Q_1(\partial,\lambda)\neq0$, which implies that $\phi_B(\lambda)=0$. Then $\phi_A(\lambda)$ must not be zero and $g(\lambda)$ is forced to be 0 by (\ref{te13}). Taking these results in (\ref{te12}) and applying Lemma \ref{r1}, we can see that $f(\lambda)$ is a scalar multiple of $\phi_A(\lambda)$ and thus the extension is trivial.
	
	Lastly, we discuss the case that $p(\lambda)\neq0$ but $Q_1(\partial,\lambda)=0$. Then $\phi_B(\lambda)=0$ and $\phi_A(\lambda)\neq 0$. In this case, (\ref{te12}) means that $f(\lambda)=k\phi_A(\lambda)$ for some $k\in\mathbb{C}$ by Lemma \ref{r1}. So if $g(\lambda)=0$, the extension is trivial. If $g(\lambda)\neq 0$, we can obtain that $\phi_A(\lambda)=-p(\lambda)$ by comparing the coefficients of the highest term of $\mu$ in  (\ref{te13}), which implies $g(\lambda)$ is a nonzero constant. By Lemma \ref{l11}, the corresponding extension is nontrivial.
	
	Since $p(\lambda)Q_1(\partial,\lambda)=0$, we have completed the proof.
\end{proof}
\subsection{$0\longrightarrow V_{\phi_A,\phi_B} \longrightarrow  E\longrightarrow\mathbb{C}c_{\eta}   \longrightarrow   0$}\hfill\\

Next, we consider extensions of finite irreducible $\mathcal{R}$-modules of the form \begin{align}\label{f12}
0\longrightarrow V_{\phi_A,\phi_B} \longrightarrow  E\longrightarrow\mathbb{C}c_{\eta}   \longrightarrow   0.
\end{align}Then $E$ is isomorphic to $V_{\phi_A,\phi_B}\oplus\mathbb{C}c_{\eta} =\mathbb{C}[\partial]v\oplus\mathbb{C}c_{\eta} $ as a vector space, and the following identities hold in $E$:
\begin{align}\label{se2}
&A_\lambda v=\phi_A(\lambda)v,\quad B_\lambda v=\phi_B(\lambda)v, \notag\\
&A_\lambda c_{\eta}=f(\partial,\lambda)v,\quad B_\lambda c_{\eta}=g(\partial,\lambda)v,\quad\partial c_{\eta}=\eta c_{\eta}+h(\partial)v, 
\end{align}
where $f(\partial,\lambda), g(\partial,\lambda), h(\partial)\in \mathbb{C}[\partial,\lambda]$.

\begin{lemma}\label{l12}
	All trivial extensions of finite irreducible $\mathcal{R}$-modules of the form (\ref{f12}) are given by (\ref{se2}), and $f(\partial,\lambda)=\varphi(\partial+\lambda)\phi_A(\lambda),g(\partial,\lambda)=\varphi(\partial+\lambda)\phi_B(\lambda)$ and $h(\partial)=(\partial-\eta)\varphi(\partial)$, where $\varphi$ is a polynomial.
\end{lemma}
\begin{proof}
	Assume that (\ref{f12}) is a trivial extension, that is, there exists $c'_\eta=kc_\eta+l(\partial)v\in E$, where $0\neq k\in\mathbb{C}$ and $ l(\partial)\in\mathbb{C}[\partial]$, such that $A_\lambda c'_\eta=B_\lambda c'_\eta=0$ and $\partial c'_{\eta}=\eta c'_{\eta}=k\eta c_\eta+\eta l(\partial)v$.
	
	On the other hand, it follows from (\ref{se2}) that
	\begin{align*}
	&A_\lambda c'_\eta=(kf(\partial,\lambda)+l(\partial+\lambda)\phi_A(\lambda))v,\\
	&B_\lambda c'_\eta=(kg(\partial,\lambda)+l(\partial+\lambda)\phi_B(\lambda))v,\\
	&\partial c'_{\eta}=k\eta c_\eta+(kh(\partial)+\partial l(\partial))v.
	\end{align*}
	We can obtain the result by comparing both expressions for $A_\lambda c'_\eta,B_\lambda c'_\eta$ and $\partial c'_\eta$.
	
	Conversely, if $f(\partial,\lambda)=\varphi(\partial+\lambda)\phi_A(\lambda),g(\partial,\lambda)=\varphi(\partial+\lambda)\phi_B(\lambda)$ and $h(\partial)=(\partial-\eta)\varphi(\partial)$ for some polynomial $\varphi$, setting $c'_{\eta}=c_\eta -\varphi(\partial)v$ we can deduce that (\ref{f12}) is a trivial extension.
\end{proof}

\begin{theorem}\label{t12}
	For solvable rank two Lie conformal algebra $\mathcal{R}$, nontrivial extensions of finite irreducible conformal modules of the form (\ref{f12}) do not exist, that is, $dim(Ext(\mathbb{C}c_\eta,V_{\phi_A,\phi_B}))=0$.
\end{theorem}
\begin{proof}
	Applying both sides of (\ref{m14}) and (\ref{m15})  to $c_\eta$ and comparing the corresponding efficients gives the following equations \begin{align}
	&(\partial+\lambda-\eta)f(\partial,\lambda)=h(\partial+\lambda)\phi_A(\lambda),\label{te21}\\
	&(\partial+\lambda-\eta)g(\partial,\lambda)=h(\partial+\lambda)\phi_B(\lambda).\label{te22}
	\end{align}
	Since $\phi_A(\lambda)$ and $\phi_B(\lambda)$ are not all equal to zero, the above equations imply that there exists a polynomial $\varphi$ such that $h(\partial)=(\partial-\eta)\varphi(\partial)$. Then we have $f(\partial,\lambda)=\varphi(\partial+\lambda)\phi_A(\lambda)$ and $g(\partial,\lambda)=\varphi(\partial+\lambda)\phi_B(\lambda)$. By Lemma \ref*{l12}, extensions of finite irreducible $\mathcal{R}$-modules of the form (\ref{f12}) are always trivial. 
\end{proof}
\subsection{$0\longrightarrow V_{\phi_A,\phi_B} \longrightarrow  E\longrightarrow  V_{\bar{\phi}_A,\bar{\phi}_B} \longrightarrow   0$}\hfill\\

Finally, we consider extensions of finite irreducible $\mathcal{R}$-modules of the form \begin{align}\label{f13}
0\longrightarrow V_{\phi_A,\phi_B} \longrightarrow  E\longrightarrow  V_{\bar{\phi}_A,\bar{\phi}_B} \longrightarrow   0.
\end{align}Then $E$ is isomorphic to $V_{\phi_A,\phi_B}\oplus V_{\bar{\phi}_A,\bar{\phi}_B} =\mathbb{C}[\partial]v\oplus\mathbb{C}[\partial]\bar{v} $ as a vector space, and the following identities hold in $E$:
\begin{align}\label{se3}
&A_\lambda v=\phi_A(\lambda)v,\quad B_\lambda v=\phi_B(\lambda)v, \notag\\
&A_\lambda \bar{v}=\bar{\phi}_A(\lambda)\bar{v}+f(\partial,\lambda)v,\quad B_\lambda \bar{v}=\bar{\phi}_B(\lambda)\bar{v}+g(\partial,\lambda)v, 
\end{align}
where $f(\partial,\lambda), g(\partial,\lambda)\in \mathbb{C}[\partial,\lambda]$.

\begin{lemma}\label{l13}
	All trivial extensions of finite irreducible $\mathcal{R}$-modules of the form (\ref{f13}) are given by (\ref{se3}), and $f(\partial,\lambda)=\varphi(\partial+\lambda)\phi_A(\lambda)-\varphi(\partial)\bar{\phi}_A(\lambda)$ and $g(\partial,\lambda)=\varphi(\partial+\lambda)\phi_B(\lambda)-\varphi(\partial)\bar{\phi}_B(\lambda)$ for some polynomial $\varphi$.
\end{lemma}
\begin{proof}
	Assume that (\ref{f13}) is a trivial extension, that is, there exists $\bar{v}'=k(\partial)v+l(\partial)\bar{v}\in E$, where $k(\partial),l(\partial)\in\mathbb{C}[\partial]$ and $ l(\partial)\neq0$, such that \begin{align*}
	&A_\lambda \bar{v}'=\bar{\phi}_A(\lambda)\bar{v}'=k(\partial)\bar{\phi}_A(\lambda)v+l(\partial)\bar{\phi}_A(\lambda)\bar{v},\\
	&B_\lambda \bar{v}'=\bar{\phi}_B(\lambda)\bar{v}'=k(\partial)\bar{\phi}_B(\lambda)v+l(\partial)\bar{\phi}_B(\lambda)\bar{v}.
	\end{align*}
	
	On the other hand, it follows from (\ref{se3}) that
	\begin{align*}
	&A_\lambda \bar{v}'=(k(\partial+\lambda)\phi_A(\lambda)+l(\partial+\lambda)f(\partial,\lambda))v+l(\partial+\lambda)\bar{\phi}_A(\lambda)\bar{v},\\
	&B_\lambda \bar{v}'=(k(\partial+\lambda)\phi_B(\lambda)+l(\partial+\lambda)g(\partial,\lambda))v+l(\partial+\lambda)\bar{\phi}_B(\lambda)\bar{v}.
	\end{align*}
	Comparing both expressions for $A_\lambda \bar{v}'$ and $B_\lambda \bar{v}'$, we can obtain that $l(\partial)$ is a nonzero constant. Then we can give the expressions of $f(\partial,\lambda)$ and $g(\partial,\lambda)$.
	
	Conversely, if $f(\partial,\lambda)=\varphi(\partial+\lambda)\phi_A(\lambda)-\varphi(\partial)\bar{\phi}_A(\lambda)$ and $g(\partial,\lambda)=\varphi(\partial+\lambda)\phi_B(\lambda)-\varphi(\partial)\bar{\phi}_B(\lambda)$ for some polynomial $\varphi$, setting $\bar{v}'=-\varphi(\partial)v+\bar{v}$ we can deduce that (\ref{f13}) is a trivial extension.
\end{proof}

To better characterize the classification procedure of nontrivial extensions, we advance part of the computation in the following lemmas.

\begin{lemma}\label{r2}
	The solutions of the equation
	\begin{align}\label{r2e}
	c(\partial,\lambda)b(\mu)+c(\partial+\lambda,\mu)a(\lambda)-c(\partial,\mu)b(\lambda)-c(\partial+\mu,\lambda)a(\mu)=0
	\end{align} for unknown polynomial $c(\partial,\lambda)\in\mathbb{C}[\partial,\lambda]$ are given as follows.
	\begin{itemize}
		\item If $a(\lambda)=b(\lambda)=0$, the equation holds for any polynomial in $\mathbb{C}[\partial,\lambda]$.
		\item If $a(\lambda)=b(\lambda)\neq0$, then $c(\partial,\lambda)= a(\lambda)(\varphi_1(\partial+\lambda)-\varphi_1(\partial))+\varphi_2(\lambda)$, where $\varphi_1,\varphi_2$ are polynomials.
		\item If $a(\lambda)\neq b(\lambda)$, then $c(\partial,\lambda)= a(\lambda)\varphi(\partial+\lambda)-b(\lambda)\varphi(\partial)$ for some polynomial $\varphi$.
	\end{itemize}
\end{lemma}
\begin{proof}
	If $a(\lambda)=b(\lambda)=0$ or $c(\partial,\lambda)=0$, the result is obvious. 
	
	Now we assume $a(\lambda)=b(\lambda)\neq0$, $c(\partial,\lambda)\neq 0$, and $c(\partial,\lambda)=\sum_{i=0}^{m}c_i(\lambda)\partial^i$ with $c_m(\lambda)\neq0$. The result can be obtained by induction on $m$. When $m=0$, the variation of (\ref{r2e})
	\begin{align}\label{r2ve}
	(c(\partial+\lambda,\mu)-c(\partial,\mu))a(\lambda)=(c(\partial+\mu,\lambda)-c(\partial,\lambda))a(\mu)
	\end{align} implies the original equation is established. Assume the conclusion holds for $m=n$ ($n\geq0$) and consider the case that $m=n+1$. Comparing the coefficients of $\partial^n$ of (\ref{r2ve}), we have $$\lambda c_{n+1}(\mu)a(\lambda)=\mu c_{n+1}(\lambda)a(\mu).$$ So $c_{n+1}(\lambda)=k\lambda a(\lambda)$ for some nonzero constant $k$ by Lemma \ref{r1}. Let $k_{n+1}=\frac{k}{n+2},d(\partial,\lambda)=k_{n+1}a(\lambda)((\partial+\lambda)^{n+2}-\partial^{n+2})$ and $e(\partial,\lambda)=c(\partial,\lambda)-d(\partial,\lambda)$. Since both $c$ and $d$ satisfy (\ref{r2ve}), by induction, $e(\partial,\lambda)=a(\lambda)(\varphi_0(\partial+\lambda)-\varphi_0(\partial))+\varphi_2(\lambda)$ with polynomials $\varphi_0,\varphi_2$. Set $\varphi_1(\partial)=\varphi_0(\partial)+k_{n+1}\partial^{n+2}$, and the expression of $c(\partial,\lambda)$ follows.
	
	For $a(\lambda)\neq b(\lambda)$, we set $c(\partial,\lambda)=\sum_{i=0}^{m}c_i(\lambda)\partial^i$ with $c_m(\lambda)\neq0$. If $m=0$, the equation (\ref{r2e}) can be rewritten as $$c_0(\lambda)(a(\mu)-b(\mu))-c_0(\mu)(a(\lambda)-b(\lambda))=0,$$ which means $c(\partial,\lambda)=k(a(\lambda)-b(\lambda))$ for some nonzero constant $k$. Thus the conclusion holds for $m=0$. Assume the conclusion holds for $m=n$ ($n\geq0$) and consider the case that $m=n+1$. The coefficients of $\partial^{n+1}$ of the two sides of (\ref{r2e}) imply $c_{n+1}(\lambda)=t(a(\lambda)-b(\lambda))$ with $t\neq0$. Let $\varphi_1(\partial)=t\partial^{n+1}$, $d(\partial,\lambda)=a(\lambda)\varphi_1(\partial+\lambda)-b(\lambda)\varphi_1(\partial)$ and $e(\partial,\lambda)=c(\partial,\lambda)-d(\partial,\lambda)$. Then $d(\partial,\lambda)$ satisfies (\ref{r2e}), and so does $e(\partial,\lambda)$. By the assumption, we have $e(\partial,\lambda)= a(\lambda)\varphi_2(\partial+\lambda)-b(\lambda)\varphi_2(\partial)$ for some polynomial $\varphi_2$. Setting $\varphi(\partial)=\varphi_1(\partial)+\varphi_2(\partial)$, one can find that the conclusion also holds for $m=n+1$.
\end{proof}
\begin{lemma}\label{lq}
	Let $f(\lambda,\mu)$ be a nonzero polynomial in $\mathbb{C}[\lambda,\mu]$ satisfying $f(\lambda,\mu)=-f(\mu,\lambda)$. Denote the coefficient of $\lambda^i\mu^j$ in $f$ by $f_{ij}$ and the antisymmetric matrix consisting of $f_{ij}$'s by $M$. Assume $f_{i_0,j_0}\neq0$ ($i_0<j_0$). Then the following statements are equivalent.
	\begin{enumerate}[(i)]
		\item There exist polynomials $P_1,P_2$ such that \begin{align}\label{ld1}
		P_1(\lambda)P_2(\mu)-P_1(\mu)P_2(\lambda)=f(\lambda,\mu).
		\end{align}
		\item For any $i,j,k,l$, $f_{ij}f_{kl}-f_{ik}f_{jl}+f_{jk}f_{il}=0$.
		\item $rank(M)=2$.
	\end{enumerate} 
\end{lemma}
\begin{proof}
	(i)$\Rightarrow$(ii). Assume that (\ref{ld1}) has been established and the expressions of $P_1,P_2$ are given by $P_1(\lambda)=\sum_{i=0}^np_{i1}\lambda^i,\ P_2(\lambda)=\sum_{i=0}^np_{i2}\lambda^i.$ Let $\mathbf{P}_1=(	p_{01},p_{11},\cdots,p_{n1})^T$ and $\mathbf{P}_2=(	p_{02},p_{12},\cdots,p_{n2})^T$. Then (ii) follows from $f_{ij}=p_{i1}p_{j2}-p_{j1}p_{i2},\forall i,j$.
	
	(ii)$\Rightarrow$(iii). Assume (ii) and perform the following elementary row and column transformations on $M$ as follows.
	\begin{align*}
	M\xrightarrow[\mathbf{c}_k-\frac{f_{k-1,j_0}}{f_{i_0,j_0}}\mathbf{c}_{{i_0+1}}+\frac{f_{k-1,i_0}}{f_{i_0,j_0}}\mathbf{c}_{{j_0+1}},\forall k\neq i_0+1, j_0+1]{\mathbf{r}_k-\frac{f_{k-1,j_0}}{f_{i_0,j_0}}\mathbf{r}_{{i_0+1}}+\frac{f_{k-1,i_0}}{f_{i_0,j_0}}\mathbf{r}_{{j_0+1}},\forall k\neq i_0+1, j_0+1}\xrightarrow[\mathbf{c}_{i_0+1}\leftrightarrow\mathbf{c}_1]{\mathbf{r}_{i_0+1}\leftrightarrow\mathbf{r}_1}\xrightarrow[\mathbf{c}_{j_0+1}\leftrightarrow\mathbf{c}_2]{\mathbf{r}_{j_0+1}\leftrightarrow\mathbf{r}_2}diag\left\{\begin{pmatrix}
	0&f_{i_0,j_0}\\
	-f_{i_0,j_0}&0
	\end{pmatrix},0,\cdots,0\right\}.
	\end{align*} Thus $rank(M)=2$.
	
	(iii)$\Rightarrow$(i). Assume that $rank(M)=2$ and denote the order of $M$ by $m$ ($m\geq2$). Then there exists an invertible matrix $P=\{q_{ij}\}$ of order $m$ such that $$M=Pdiag\left\{\begin{pmatrix}
	0&1\\
	-1&0
	\end{pmatrix},0,\cdots,0\right\}P^T.$$ Define $P_1(\lambda)=\sum_{i=0}^{m-1}q_{i+1,1}\lambda^i,\ P_2(\lambda)=\sum_{i=0}^{m-1}q_{i+1,2}\lambda^i$. It is easy to check that $P_1,P_2$ satisfy \ref{ld1}.
\end{proof}

	For $f(\lambda,\mu)$ meeting the condition in Lemma \ref{lq}, based on the proof of (ii)$\Rightarrow$(iii) and (iii)$\Rightarrow$(i), we can write down a pair of polynomials $P_1,P_2$ satisfying (\ref{ld1}) as follows: \begin{align}\label{ld2}
	P_1(\lambda)=\sum_kf_{k,j_0}\lambda^k,\ P_2(\lambda)=-\frac{1}{f_{i_0,j_0}}\sum_kf_{k,i_0}\lambda^k.
	\end{align} With the above discussion, we can solve the following equation. 

\begin{lemma}\label{lq1}
	Let $f(\lambda,\mu)$ be a polynomial satisfying the condition in Lemma \ref{lq} with $P_1,P_2$ defined in (\ref{ld2}). The equation
	\begin{align}\label{ld3}
		a(\lambda)b(\mu)-a(\mu)b(\lambda)=f(\lambda,\mu)
	\end{align} for unknown polynomial $b(\partial)\in\mathbb{C}[\partial]$ has solutions only when $a(\lambda)=\frac{a_{i_0}}{f_{i_0,j_0}}P_1(\lambda)+a_{j_0}P_2(\lambda)$ and not all $a_{i_0},a_{j_0}$ equal 0. If solutions do exist, they are given as follows.
	\begin{itemize}
		\item If $a_{i_0}\neq0$, then $b(\lambda)=ka(\lambda)+\frac{f_{i_0,j_0}}{a_{i_0}}P_2(\lambda)$, where $k\in\mathbb{C}$.
		\item If $a_{i_0}=0,a_{j_0}\neq0$, then $b(\lambda)=ka(\lambda)-\frac{1}{a_{j_0}}P_1(\lambda)$, where $k\in\mathbb{C}$.
	\end{itemize} 
\end{lemma}
\begin{proof}
	First we consider the special case that the coefficient matrix $M$ of $f$ is $$diag\left\{\begin{pmatrix}
	0&1\\
	-1&0
	\end{pmatrix},0,\cdots,0\right\}.$$ In this case, $P_1(\lambda)=1,P_2(\lambda)=\lambda$.
	Denote the order of $M$ by $n+1$ and write $a(\lambda)=\sum_{i=0}^na_i\lambda^i,b(\lambda)=\sum_{i=0}^nb_i\lambda^i$. Taking them in (\ref{ld3}), we have \begin{align}\label{lde}
	\begin{cases}
	a_0b_1-a_1b_0=1,\\
	a_ib_j-a_jb_i=0, \forall i<j,(i,j)\neq(0,1).
	\end{cases}
	\end{align} So the equation has solutions only when not all $a_0,a_1$ are equal to 0. Assume the solutions do exist. If $a_0\neq0$, (\ref{lde}) implies $b_1=\frac{1+a_1b_0}{a_0},b_j=\frac{a_jb_0}{a_0},\forall j>1$ and then $a_j=0,\forall j>1$. Thus, $a(\lambda)=a_0+a_1\lambda$ and $b(\lambda)=b_0+\frac{1+a_1b_0}{a_0}\lambda$. If $a_0=0$, then (\ref{lde}) implies $a_1\neq0,b_0=-\frac{1}{a_1}$ and $a_j,b_j=0, \forall j>1$. Therefore, $a(\lambda)=a_1\lambda$ and $b(\lambda)=-\frac{1}{a_1}+b_1\lambda$.
	
	For the general case, we use $\mathbf{a}=(a_0,a_1,\cdots,a_n)^T,\mathbf{b}=(b_0,b_1,\cdots,b_n)^T$ to mean the coefficient matrices of $a(\lambda),b(\lambda)$ and $P$ to mean the invertible matrix corresponding to $P_1(\lambda),P_2(\lambda)$ as mentioned in the proof of (iii)$\Rightarrow$(i). Denote $\mathbf{a}'=P^{-1}\mathbf{a},\mathbf{b}'=P^{-1}\mathbf{b}$. Since $M=\mathbf{a}\mathbf{b}^T-\mathbf{b}\mathbf{a}^T$, we have $$diag\left\{\begin{pmatrix}
	0&1\\
	-1&0
	\end{pmatrix},0,\cdots,0\right\}=\mathbf{a}'{\mathbf{b}'}^T-\mathbf{b}'{\mathbf{a}'}^T.$$ By the above discussion, $\mathbf{b}'$ exists only when $\mathbf{a}'=({a'}_0,{a'}_1,0,\cdots,0)$ with ${a'}_0,{a'}_1$ not being 0 simultaneously, which is equivalent to  $a(\lambda)={a'}_0P_1(\lambda)+{a'}_1P_2(\lambda)$ with $({a'}_0,{a'}_1)\neq(0,0)$. Focusing on the coefficient of $\lambda^{i_0},\lambda^{j_0}$, we have ${a'}_0=\frac{a_{i0}}{f_{i_0,j_0}},{a'}_1=a_{j_0}$. Similarly, we can obtain the general expression of $b(\lambda)$ in different cases. 
\end{proof}
\begin{theorem}\label{t13}
	For solvable rank two Lie conformal algebra $\mathcal{R}$, nontrivial extensions of finite irreducible conformal modules of the form (\ref{f13}) always exist. Moreover, they are given (up to equivalence) by (\ref{se3}). The corresponding polynomials $\phi_A(\lambda),\phi_B(\lambda),\bar{\phi}_A(\lambda),\bar{\phi}_B(\lambda),f(\partial,\lambda)$ and $g(\partial,\lambda)$ giving rise to nontrivial extensions, are listed as follows:
	\begin{enumerate}
		\item In the case that $p(\lambda)=Q_1(\partial,\lambda)=0$.	\begin{enumerate}[(i)]
			\item If $\phi_A(\lambda)=\bar{\phi}_A(\lambda)=0,\phi_B(\lambda)=\bar{\phi}_B(\lambda)\neq0$, then $f(\partial,\lambda)=s(\lambda), g(\partial,\lambda)= t(\lambda)$, where $s,t$ are polynomials, and either $s\neq0$ or $t(\lambda)$ is not a scalar multiple of $\lambda\phi_B(\lambda)$.
			\item If $\phi_A(\lambda)=\bar{\phi}_A(\lambda)\neq0,\phi_B(\lambda)=\bar{\phi}_B(\lambda)=0$, then $f(\partial,\lambda)= s(\lambda), g(\partial,\lambda)=t(\lambda)$, where $s,t$ are polynomials, and either $t\neq0$ or $s(\lambda)$ is not a scalar multiple of $\lambda\phi_A(\lambda)$.
			\item If $\phi_A(\lambda)=\bar{\phi}_A(\lambda)\neq0,\phi_B(\lambda)=\bar{\phi}_B(\lambda)\neq0$, then $f(\partial,\lambda)= s(\lambda), g(\partial,\lambda)= t(\lambda)$, where $s,t$ are polynomials, and $s(\lambda),t(\lambda)$ are not the same scalar multiple of $\lambda\phi_A(\lambda),\lambda\phi_B(\lambda)$ respectively.
		\end{enumerate}
		\item In the case that $p(\lambda)=0,Q_1(\partial,\lambda)\neq0$, we always have $\phi_B(\lambda)=\bar{\phi}_B(\lambda)=0$. 
		\begin{enumerate}[(i)]
			\item If $\phi_A(\lambda)=\bar{\phi}_A(\lambda)\neq0, g(\partial,\lambda)= 0$, then $f(\partial,\lambda)= s(\lambda)$, where the polynomial $s(\lambda)$ is not a scalar multiple of $\lambda\phi_A(\lambda)$.
			\item If $\phi_A(\lambda)=\bar{\phi}_A(\lambda)\neq0, g(\partial,\lambda)=t(\lambda)\neq 0$ such that the coefficient matrix $M=\{q_{ij}\}$ of $Q_1(-\lambda-\mu,\lambda)t(\lambda+\mu)$ is of rank 2 and for $q_{i_0,j_0}\neq 0$, $\phi_A(\lambda)=\frac{a_{i0}}{f_{i_0,j_0}}(\sum_kf_{k,j_0}\lambda^k)-\frac{a_{j_0}}{f_{i_0,j_0}}(\sum_kf_{k,i_0}\lambda^k)$ with the coefficients $a_{i_0-1},a_{j_0-1}$ of $\lambda^{i_0-1},\lambda^{j_0-1}$ in $\phi_A$ are not all 0, then $$f(\partial,\lambda)=\begin{cases}
			-\frac{1}{a_{i_0-1}}(\sum_{k}q_{k,j_0}\lambda^k)\partial+s(\lambda),&\mbox{ if }a_{i_0-1}\neq0,\\
			-\frac{1}{a_{j_0-1}}(\sum_{k}q_{k,i_0}\lambda^k)\partial+s(\lambda),&\mbox{ if }a_{i_0-1}=0,
			\end{cases}$$ where $s$ is a polynomial.
		\end{enumerate}
		\item In the case that $p(\lambda)\neq0,Q_1(\partial,\lambda)=0$, we have $\phi_B(\lambda)=\bar{\phi}_B(\lambda)=0$.		\begin{enumerate}[(i)]
			\item If $\phi_A(\lambda)=\bar{\phi}_A(\lambda)\neq0$,  then $f(\partial,\lambda)= s(\lambda), g(\partial,\lambda)= 0$, where $s_1,s_2$ are polynomials, and $s(\lambda)$ is not a scalar multiple of $\lambda\phi_A(\lambda)$.
			\item If $\phi_A(\lambda)\neq\bar{\phi}_A(\lambda)$, then $f(\partial,\lambda)=0,\phi_A(\lambda)-\bar{\phi}_A(\lambda)=p(\lambda)$, and \begin{align*}
			g(\partial,\lambda)=\begin{cases}
			k_1(\partial+\frac{1}{r}\lambda)+k_2,&p(\lambda)=r\phi_A(\lambda)\mbox{ and }r\neq1,\\
			k_1,&p(\lambda)\mbox{ is not a scalar multiple of }\phi_A(\lambda),
			\end{cases}
			\end{align*} where $k_1,k_2\in\mathbb{C}$ and $g(\partial,\lambda)\neq0$.
		\end{enumerate}
	\end{enumerate}
The space of $Ext(V_{\bar{\phi}_A,\bar{\phi}_B},V_{\phi_A,\phi_B})$ is of infinite dimension in all of the above subcases but (3)-(ii).
\end{theorem}
\begin{proof}
	Applying both sides of (\ref{m11}), (\ref{m12}) and (\ref{m13}) to $\bar{v}$ and comparing the corresponding efficients, we obtain \begin{align}
	&Q_1(-\lambda-\mu,\lambda)\bar{\phi}_B(\lambda+\mu)=p(\lambda)\bar{\phi}_B(\lambda+\mu)=0,\label{te31}\\
	&f(\partial,\lambda)\bar{\phi}_A(\mu)+f(\partial+\lambda,\mu){\phi_A}(\lambda)-f(\partial,\mu)\bar{\phi}_A(\lambda)-f(\partial+\mu,\lambda){\phi_A}(\mu)=Q_1(-\lambda-\mu,\lambda)g(\partial,\lambda+\mu),\label{te32}\\
	&f(\partial,\lambda)\bar{\phi}_B(\mu)+g(\partial+\lambda,\mu){\phi_A}(\lambda)-g(\partial,\mu)\bar{\phi}_A(\lambda)-f(\partial+\mu,\lambda){\phi_B}(\mu)=p(\lambda)g(\partial,\lambda+\mu),\label{te33}\\
	&g(\partial,\lambda)\bar{\phi}_B(\mu)+g(\partial+\lambda,\mu){\phi_B}(\lambda)-g(\partial,\mu)\bar{\phi}_B(\lambda)-g(\partial+\mu,\lambda){\phi_B}(\mu)=0.\label{te34}
	\end{align}
	
	\textbf{Case 1.} $p(\lambda)=Q_1(\partial,\lambda)=0$.
	
	(i) If $\phi_A(\lambda)=\bar{\phi}_A(\lambda)=0,\phi_B(\lambda)=\bar{\phi}_B(\lambda)\neq0$, we can obtain $f(\partial,\lambda)\in\mathbb{C}[\lambda]$ from (\ref{te33}).
		And by (\ref{te34}) and Lemma \ref{r2}, we have $g(\partial,\lambda)= \phi_B(\lambda)(t_1(\partial+\lambda)-t_1(\partial))+t_2(\lambda)$ for some polynomials $t_1,t_2$. According to Lemma \ref{l13}, this extension is equivalent to the extension with the same $f(\partial,\lambda)$ and $g(\partial,\lambda)=t_2(\lambda)$, and is nontrivial only if $f(\partial,\lambda)\neq0$ or $t_2(\lambda)$ is not a scalar multiple of $\lambda\phi_B(\lambda)$.
		
	(ii) 	If $\phi_A(\lambda)=\bar{\phi}_A(\lambda)=0,\phi_B(\lambda)\neq\bar{\phi}_B(\lambda)$, (\ref{te33}) implies $f(\partial,\lambda)=0$ by comparing the coefficients of the highest order term with respect to $\partial$. Meanwhile, (\ref{te34}) means $g(\partial,\lambda)= \phi_B(\lambda)t(\partial+\lambda)-\bar{\phi}_B(\lambda)t(\partial)$ for some polynomial $t$ by Lemma \ref{r2}. In this subcase, the extension is trivial.
		
(iii) 	If $\phi_A(\lambda)=\bar{\phi}_A(\lambda)\neq0,\phi_B(\lambda)=\bar{\phi}_B(\lambda)\neq0$, we can deduce that $f(\partial,\lambda)= \phi_A(\lambda)(s_1(\partial+\lambda)-s_1(\partial))+s_2(\lambda), g(\partial,\lambda)= \phi_B(\lambda)(t_1(\partial+\lambda)-t_1(\partial))+t_2(\lambda)$ for some polynomials $s_1,s_2,t_1,t_2$ from (\ref{te32}) and (\ref{te34}). Taking them in (\ref{te33}) and setting $r(\lambda)=t_1(\lambda)-s_1(\lambda)$, we can obtain $$r(\partial+\lambda+\mu)-r(\partial+\lambda)-r(\partial+\mu)+r(\partial)=0,$$ which implies $r(\lambda)=r_1\lambda+r_0$ for some $r_0,r_1\in\mathbb{C}$. This extension is equivalent to the extension with $f(\partial,\lambda)=s_2(\lambda)$ and $g(\partial,\lambda)=t'(\lambda)$ where $t'(\lambda)=r_1\lambda\phi_B(\lambda)+t_2(\lambda)$. By Lemma \ref{l13}, that the extension is nontrivial requires that $s_2(\lambda),t'(\lambda)$ are not the same scalar multiple of $\lambda\phi_A(\lambda),\lambda\phi_B(\lambda)$ respectively.
	
	(iv) 	If $\phi_A(\lambda)=\bar{\phi}_A(\lambda)\neq0,\phi_B(\lambda)\neq\bar{\phi}_B(\lambda)$, then $f(\partial,\lambda)= \phi_A(\lambda)(s_1(\partial+\lambda)-s_1(\partial))+s_2(\lambda), g(\partial,\lambda)= \phi_B(\lambda)t(\partial+\lambda)-\bar{\phi}_B(\lambda)t(\partial)$ for some polynomials $s_1,s_2,t$ from (\ref{te32}) and (\ref{te34}).  Taking them in (\ref{te33}) and setting $r(\lambda)=t(\lambda)-s_1(\lambda)$, we can obtain $$\phi_A(\lambda)\phi_B(\mu)(r(\partial+\lambda+\mu)-r(\partial+\mu))-\phi_A(\lambda)\bar{\phi}_B(\mu)(r(\partial+\lambda)-r(\partial))=s_2(\lambda)(\phi_B(\mu)-\bar{\phi}_B(\mu)).$$ Denote the degree of $r(\lambda)$ by $m$. If $m\geq2$, comparing the coefficients of $\partial^{m-1}$ on each side of the above equation, we can get a contradiction. Let $r(\lambda)=r_1\lambda+r_0$ with $r_1,r_2\in\mathbb{C}$. Then we have $s_2(\lambda)=r_1\lambda\phi_A(\lambda)$. And the extension is always trivial in this subcase because $f(\partial,\lambda)=\phi_A(\lambda)(s_1(\partial+\lambda)-s_1(\partial))=\phi_A(\lambda)(t(\partial+\lambda)-t(\partial))$.
		
	(v) 	If $\phi_A(\lambda)\neq\bar{\phi}_A(\lambda),\phi_B(\lambda)\neq\bar{\phi}_B(\lambda)$, then $f(\partial,\lambda)= \phi_A(\lambda)s(\partial+\lambda)-\bar{\phi}_A(\lambda)s(\partial), g(\partial,\lambda)= \phi_B(\lambda)t(\partial+\lambda)-\bar{\phi}_B(\lambda)t(\partial)$ for some polynomials $s,t$. If $r(\partial)=s(\partial)-t(\partial)\neq0$, then $r(\partial)$ can be written as $r(\partial)=\sum_{i=0}^mr_i\partial^i$ with $r_m\neq0$. Taking them in (\ref{te33}) and considering the coefficients of $\partial^{m}$ give \begin{align*}
		r_m(\phi_A(\lambda)-\bar{\phi}_A(\lambda))(\bar{\phi}_B(\mu)-\phi_B(\mu))=0,
		\end{align*} which contradicts the assumption. So we can deduce that $s(\partial)=t(\partial)$ and thus the extension is trivial.

The other subcases can be learned from the symmetry of $(\phi_A,\bar{\phi}_A)$ and $(\phi_B,\bar{\phi}_B)$.

	\textbf{Case 2.} $p(\lambda)=0,Q_1(\partial,\lambda)\neq0$. In this case, ${\phi_B}(\lambda)=\bar{\phi}_B(\lambda)=0$ by (\ref{te31}) and then $\phi_A(\lambda)$ and $\bar{\phi}_A(\lambda)$ are nonzero polynomials. 
	
(i) If $\phi_A(\lambda)=\bar{\phi}_A(\lambda)\neq0$, then (\ref{te33}) implies $g(\partial,\lambda)=t(\lambda)\in\mathbb{C}[\lambda]$. Put it in (\ref{te32}) and take the partial derivative of both sides of the equation with respect to $\partial$, and we can obtain \begin{align}
		(f_\partial(\partial+\lambda,\mu)-f_\partial(\partial,\mu))\phi_A(\lambda)-(f_\partial(\partial+\mu,\lambda)-f_\partial(\partial,\lambda))\phi_A(\mu)=0.
		\end{align}
		By Lemma \ref{r2}, $f_\partial(\partial,\lambda)=\phi_A(\lambda)(v_1(\partial+\lambda)-v_1(\partial))+v_2(\lambda)$, where $v_1,v_2$ are polynomials. Let $s_1(\partial)=\int v_1(\partial) d\partial$, and then $$f(\partial,\lambda)=\int f_\partial(\partial,\lambda)d\partial=\phi_A(\lambda)(s_1(\partial+\lambda)-s_1(\partial))+v_2(\lambda)\partial+v_3(\lambda),$$ where $v_3$ is a polynomial. Taking this result in (\ref{te32}) again, one can get \begin{align}\label{ve33}
		\lambda\phi_A(\lambda)v_2(\mu)-\mu\phi_A(\mu)v_2(\lambda)=Q_1(-\lambda-\mu,\lambda)t(\lambda+\mu).
		\end{align} With this equation, we have the following two subceses.
		
		 If $t=0$, $v_2(\lambda)=k\lambda a(\lambda)$ for some constant $k$. Let ${s'}_1(\partial)=s_1(\partial)+\frac{k}{2}\partial^2$. Then $f(\partial,\lambda)=\phi_A(\lambda)({s'}_1(\partial+\lambda)-{s'}_1(\partial))+s_2(\lambda)$ for some polynomials ${s'}_1,s_2$. In this case, the extension is nontrivial only if $s_2$ is not a scalar multiple of $\lambda\phi_A(\lambda)$.
		 
		If $t\neq0$, then there exists $v_2(\lambda)$ satisfying (\ref{ve33}) only when $Q(\lambda,\mu)=Q_1(-\lambda-\mu,\lambda)t(\lambda+\mu),\lambda\phi_A(\lambda)$ meet the condition in Lemma \ref{lq} and \ref{lq1}. Under these conditions, we can give the expression of $v_2(\lambda)$ and then that of $f(\partial,\lambda)$. In this case, the extension is nontrivial.

(ii)	 If $\phi_A(\lambda)\neq\bar{\phi}_A(\lambda)$, then (\ref{te33}) implies $g(\partial,\lambda)=0$. So by (\ref{te32}) and Lemma \ref*{r2}, we have $f(\partial,\lambda)=\phi_A(\lambda)s(\partial+\lambda)-\bar{\phi}_A(\lambda)s(\partial)$ for some polynomial $s$. Thus the extension is trivial under the condition.

	\textbf{Case 3.} $p(\lambda)\neq0,Q_1(\partial,\lambda)=0$.   In this case, ${\phi_B}(\lambda)=\bar{\phi}_B(\lambda)=0$ by (\ref{te31}) and then $\phi_A(\lambda)$ and $\bar{\phi}_A(\lambda)$ are nonzero polynomials. 
	
(i) If $\phi_A(\lambda)=\bar{\phi}_A(\lambda)\neq0$, then $f(\partial,\lambda)=\phi_A(\lambda)(s_1(\partial+\lambda)-s_1(\partial))+s_2(\lambda)$ with polynomials $s_1,s_2$ by (\ref{te32}). Assume $g(\partial,\lambda)\neq0$. Comparing the coefficients of the highest item with respect to $\partial$ in (\ref{te33}), we get $p(\lambda)=0$, which contracts the given condition. So $g(\partial,\lambda)=0$ and the extension is nontrivial only when $s_2(\lambda)$ is not a scalar multiple of $\lambda\phi_A(\lambda)$.

		(ii) If $\phi_A(\lambda)\neq\bar{\phi}_A(\lambda)$, then $f(\partial,\lambda)=\phi_A(\lambda)s(\partial+\lambda)-\bar{\phi}_A(\lambda)s(\partial)$ with polynomial $s$ by (\ref{te32}). Assume $g(\partial,\lambda)\neq 0$. Let $g(\partial,\lambda)=\sum_{i=0}^{m}g_i(\lambda)\partial^i$. Comparing the coefficients of $\partial^m$ in (\ref{te33}), we have $p(\lambda)=\phi_A(\lambda)-\bar{\phi}_A(\lambda)$ and $g_m(\lambda)=k_1\in\mathbb{C}^\times$. If $m\geq1$, one can obtain $p(\lambda)=r\phi_A(\lambda)$ for some nonzero constant $r$ and $g_{m-1}(\lambda)=\frac{mk_1}{r}\lambda+k_2$ with $k_2\in\mathbb{C}$ by comparing the coefficients of $\partial^{m-1}$ in (\ref{te32}). If $m\geq2$, the coefficients of $\partial^{m-2}$ imply $r=1$ and then $\bar{\phi}_A(\lambda)=0$. Thus we get a contradiction. The extension is nontrivial if and only if $g(\partial,\lambda)\neq 0$.

\end{proof}
	\section{For rank two Lie conformal algebras that are of Type I}\label{s5}
	Let $\mathcal{R}$ be the conformal algebra defined in (\ref{a2}). Then there is a basis $\{A,B\}$ such that
	\begin{equation}
	[A_\lambda A]=(\partial+2\lambda)A,\ [A_\lambda B]=0,\ [B_\lambda B]=0.
	\end{equation}
	In this section, we deal with the extension problem over $\mathcal{R}$.
	If $V$ is a non-trivial finite irreducible $\mathcal{R}$-module, then either\begin{equation}
	V\cong V_{\alpha,\beta,\phi}=\mathbb{C}[\partial]v,\quad A_\lambda v=\delta_1(\partial+\alpha\lambda+\beta)v,\quad B_\lambda v=\delta_2\phi(\lambda)v,
	\end{equation}
	where $\delta_1,\delta_2\in\{0,1\},\delta_1^2+\delta_2^2=1,\beta,0\neq\alpha\in\mathbb{C}$, and $\phi$ is a nonzero polynomial.
	
	By definition \ref{d2}, the $\mathcal{R}$-module structure on $V$ given by $A_\lambda,B_\lambda\in End_\mathbb{C}(V)[\lambda]$ satisfies
	\begin{align}
	&[A_{\lambda},A_{\mu}]=(\lambda-\mu)A_{\lambda+\mu}, \label{m31}\\
	&[A_{\lambda},B_{\mu}]=0, \label{m32}\\
	&[B_{\lambda},B_{\mu}]=0, \label{m33}\\
	&[\partial,A_{\lambda}]=-\lambda A_{\lambda}, \label{m34}\\
	&[\partial,B_{\lambda}]=-\lambda B_{\lambda}. \label{m35}
	\end{align}

	\subsection{$0\longrightarrow \mathbb{C}c_{\eta} \longrightarrow  E\longrightarrow  V_{\alpha,\beta,\phi} \longrightarrow   0$}\hfill\\

	First, we consider extensions of finite irreducible $\mathcal{R}$-modules of the form \begin{align}\label{f31}
	0\longrightarrow \mathbb{C}c_{\eta} \longrightarrow  E\longrightarrow  V_{\alpha,\beta,\phi} \longrightarrow   0.
	\end{align}Then $E$ is isomorphic to $\mathbb{C}c_{\eta}\oplus V_{\alpha,\beta,\phi}=\mathbb{C}c_{\eta}\oplus \mathbb{C}[\partial]v$ as a vector space, and the following identities hold in $E$:
	\begin{align}\label{se7}
	&\mathcal{R}_\lambda c_{\eta}=0,\quad\partial c_{\eta}=\eta c_{\eta} \notag\\
	&A_\lambda v=\delta_1(\partial+\alpha\lambda+\beta) v+f(\lambda)c_{\eta},\quad B_\lambda v=\delta_2\phi(\lambda)v+g(\lambda)c_{\eta}, 
	\end{align}
	where $f(\lambda), g(\lambda) \in \mathbb{C}[\lambda]$.
	
	\begin{lemma}\label{l31}
	All trivial extensions of finite irreducible $\mathcal{R}$-modules of the form (\ref{f31}) are given by (\ref{se7}), and $f(\lambda)$ and $g(\lambda)$ are the same scalar multiple of $\delta_1(\alpha\lambda+\eta+\beta)$ and $\delta_2\phi(\lambda)$ respectively.
		
	\end{lemma}
	\begin{proof}
		Assume that (\ref{f31}) is a trivial extension, that is, there exists $v'=kc_\eta+l(\partial)v\in E$, where $k\in\mathbb{C}$ and $0\neq l(\partial)\in\mathbb{C}[\partial]$, such that \begin{align*}
		&A_\lambda v'=\delta_1(\partial+\alpha\lambda+\beta)v'=\delta_1k(\eta+\alpha\lambda+\beta)c_\eta+\delta_1l(\partial)(\partial+\alpha\lambda+\beta)v,\\
		&B_\lambda v'=\delta_2\phi(\lambda)v'=\delta_2k\phi(\lambda)c_\eta+\delta_2l(\partial)\phi(\lambda)v.
		\end{align*}
		
		On the other hand, it follows from (\ref{se7}) that
		\begin{align*}
		&A_\lambda v'=f(\lambda)l(\eta+\lambda)c_\eta+\delta_1 l(\partial+\lambda)(\partial+\alpha\lambda+\beta)v,\\
		&B_\lambda v'=g(\lambda)l(\eta+\lambda)c_\eta+\delta_2l(\partial+\lambda)\phi(\lambda)v.
		\end{align*}
		We can obtain that $l(\partial)$ is a nonzero constant by comparing both expressions for $A_\lambda v'$ and $B_\lambda v'$. Thus $f(\lambda)$ and $g(\lambda)$ are the same scalar multiple of $\delta_1(\alpha\lambda+\eta+\beta)$ and $\delta_2\phi(\lambda)$ respectively.
		
		Conversely, if $f(\lambda)=\delta_1k(\alpha\lambda+\eta+\beta)$ and $g(\lambda)=\delta_2k\phi(\lambda)$ for some $k\in\mathbb{C}$, setting $v'=kc_\eta +v$ we can deduce that (\ref{f31}) is a trivial extension.

	\end{proof}
	
	\begin{theorem}\label{t31}
		For rank two Lie conformal algebra $\mathcal{R}$ that is of Type I, nontrivial extensions of finite irreducible conformal modules of the form (\ref{f31}) exist only when $(\delta_1,\delta_2)=(1,0),\alpha\in\{1,2\},\beta+\eta=0$. Moreover, they are given (up to equivalence) by (\ref{se7}). The values of $\eta$, along with the corresponding polynomials $f(\lambda)$ and $g(\lambda)$ giving rise to nontrivial extensions, are listed as follows: $g(\lambda)=0$ and \begin{align*}
			f(\lambda)=\begin{cases}
			s_1\lambda^2,&\alpha=1,\\
			s_2\lambda^3,&\alpha=2,
			\end{cases}
			\end{align*}
			with nonzero constants $s_1,s_2$. In these cases, $dim(Ext(V_{\alpha,\beta,\phi},\mathbb{C}c_\eta))=1$.

	\end{theorem}
	\begin{proof}
		Applying both sides of (\ref{m31}), (\ref{m32}) and (\ref{m33}) to $v$ and comparing the corresponding efficients, we obtain \begin{align}
		&\delta_1(\eta+\lambda+\alpha\mu+\beta)f(\lambda)-\delta_1(\eta+\mu+\alpha\lambda+\beta)f(\mu)=(\lambda-\mu)f(\lambda+\mu),\label{te71}\\
		&\delta_2\phi(\mu)f(\lambda)-\delta_1(\eta+\mu+\alpha\lambda+\beta)g(\mu)=0,\label{te72}\\
		&\delta_2\phi(\mu)g(\lambda)-\delta_2\phi(\lambda)g(\mu)=0.\label{te73}
		\end{align}
		
		If $(\delta_1,\delta_2)=(1,0)$, (\ref{te72}) implies $g(\mu)=0$ and it reduces to the case of Virasoro conformal algebra. We can deduce the result by Proposition 2.1 in \cite{ckw1}. If $(\delta_1,\delta_2)=(0,1)$, then $f(\lambda)=0$ by (\ref{te72}). Applying Lemma \ref{r1} to (\ref{te73}), we have $g(\lambda)=k\phi(\lambda)$ for some constant $k$ and then the extension is trivial.
	\end{proof}

\subsection{$0\longrightarrow V_{\alpha,\beta,\phi} \longrightarrow  E\longrightarrow\mathbb{C}c_{\eta}   \longrightarrow   0$}\hfill\\

Next, we consider extensions of finite irreducible $\mathcal{R}$-modules of the form \begin{align}\label{f32}
0\longrightarrow V_{\alpha,\beta,\phi} \longrightarrow  E\longrightarrow\mathbb{C}c_{\eta}   \longrightarrow   0.
\end{align}Then $E$ is isomorphic to $V_{\alpha,\beta,\phi}\oplus\mathbb{C}c_{\eta} =\mathbb{C}[\partial]v\oplus\mathbb{C}c_{\eta} $ as a vector space, and the following identities hold in $E$:
\begin{align}\label{se8}
&A_\lambda v=\delta_1(\partial+\alpha\lambda+\beta)v,\quad B_\lambda v=\delta_2\phi(\lambda)v, \notag\\
&A_\lambda c_{\eta}=f(\partial,\lambda)v,\quad B_\lambda c_{\eta}=g(\partial,\lambda)v,\quad\partial c_{\eta}=\eta c_{\eta}+h(\partial)v, 
\end{align}
where $f(\partial,\lambda), g(\partial,\lambda), h(\partial)\in \mathbb{C}[\partial,\lambda]$.

\begin{lemma}\label{l32}
	All trivial extensions of finite irreducible $\mathcal{R}$-modules of the form (\ref{f32}) are given by (\ref{se8}), and $f(\partial,\lambda)=\delta_1\varphi(\partial+\lambda)(\partial+\alpha\lambda+\beta),g(\partial,\lambda)=\delta_2\varphi(\partial+\lambda)\phi(\lambda)$ and $h(\partial)=(\partial-\eta)\varphi(\partial)$, where $\varphi$ is a polynomial.
\end{lemma}
\begin{proof}
	Assume that (\ref{f32}) is a trivial extension, that is, there exists $c'_\eta=kc_\eta+l(\partial)v\in E$, where $0\neq k\in\mathbb{C}$ and $ l(\partial)\in\mathbb{C}[\partial]$, such that $A_\lambda c'_\eta=B_\lambda c'_\eta=0$ and $\partial c'_{\eta}=\eta c'_{\eta}=k\eta c_\eta+\eta l(\partial)v$.
	
	On the other hand, it follows from (\ref{se8}) that
	\begin{align*}
	&A_\lambda c'_\eta=(kf(\partial,\lambda)+\delta_1l(\partial+\lambda)(\partial+\alpha\lambda+\beta))v,\\
	&B_\lambda c'_\eta=(kg(\partial,\lambda)+\delta_2l(\partial+\lambda)\phi(\lambda))v,\\
	&\partial c'_{\eta}=k\eta c_\eta+(kh(\partial)+\partial l(\partial))v.
	\end{align*}
	We can obtain the result by comparing both expressions for  $A_\lambda c'_\eta,B_\lambda c'_\eta$ and $\partial c'_\eta$.
	
	Conversely, if $f(\partial,\lambda)=\delta_1\varphi(\partial+\lambda)(\partial+\alpha\lambda+\beta),g(\partial,\lambda)=\delta_2\varphi(\partial+\lambda)\phi(\lambda)$ and $h(\partial)=(\partial-\eta)\varphi(\partial)$ for some polynomial $\varphi$, setting $c'_{\eta}=c_\eta -\varphi(\partial)v$, we can deduce that (\ref{f32}) is a trivial extension. 
\end{proof}

\begin{theorem}\label{t32}
	For rank two Lie conformal algebra $\mathcal{R}$ that is of Type I, nontrivial extensions of finite irreducible conformal modules of the form (\ref{f32}) exist only when $\delta_1=1,\alpha=1,\beta+\eta=0$. Moreover, the space of $Ext(\mathbb{C}c_\eta,V_{\alpha,\beta,\phi})$ is 1-dimensional, and the unique nontrivial extension is given (up to equivalence) as follows: $\delta_2=0,g(\partial,\lambda)=0$ and $f(\partial,\lambda)=h(\partial)=s$ with nonzero constant $s$. 		
\end{theorem}
\begin{proof}
	Applying both sides of (\ref{m34}) and (\ref{m35})  to $c_\eta$ and comparing the corresponding efficients gives the following equations \begin{align}
	&(\partial+\lambda-\eta)f(\partial,\lambda)=\delta_1h(\partial+\lambda)(\partial+\alpha\lambda+\beta),\label{te81}\\
	&(\partial+\lambda-\eta)g(\partial,\lambda)=\delta_2h(\partial+\lambda)\phi(\lambda).\label{te82}
	\end{align}
	
	If $(\delta_1,\delta_2)=(1,0)$, then $g=0$ by (\ref{te82}) and the result can be deduced by Proposition 2.2 in \cite{ckw1}. If $(\delta_1,\delta_2)=(0,1)$, then $f=0$. (\ref{te82}) and Lemma \ref{l32} implies that the extension is trivial in this case.
\end{proof}
\subsection{$0\longrightarrow V_{\alpha,\beta,\phi} \longrightarrow  E\longrightarrow  V_{\bar{\alpha},\bar{\beta},\bar{\phi}} \longrightarrow   0$}\hfill\\

Finally, we consider extensions of finite irreducible $\mathcal{R}$-modules of the form \begin{align}\label{f33}
0\longrightarrow V_{\alpha,\beta,\phi} \longrightarrow  E\longrightarrow  V_{\bar{\alpha},\bar{\beta},\bar{\phi}} \longrightarrow   0.
\end{align}Then $E$ is isomorphic to $V_{\alpha,\beta,\phi} \oplus  V_{\bar{\alpha},\bar{\beta},\bar{\phi}} =\mathbb{C}[\partial]v\oplus\mathbb{C}[\partial]\bar{v} $ as a vector space, and the following identities hold in $E$:
\begin{align}\label{se9}
&A_\lambda v=\delta_1(\partial+\alpha\lambda+\beta)v,\quad B_\lambda v=\delta_2\phi(\lambda)v, \notag\\
&A_\lambda \bar{v}=\bar{\delta}_1(\partial+\bar{\alpha}\lambda+\bar{\beta})\bar{v}+f(\partial,\lambda)v,\quad B_\lambda \bar{v}=\bar{\delta}_2\bar{\phi}(\lambda)\bar{v}+g(\partial,\lambda)v, 
\end{align}
where $f(\partial,\lambda), g(\partial,\lambda)\in \mathbb{C}[\partial,\lambda]$.
\begin{lemma}\label{l33}
	All trivial extensions of finite irreducible $\mathcal{R}$-modules of the form (\ref{f33}) are given by (\ref{se9}), and $f(\partial,\lambda)=\delta_1\varphi(\partial+\lambda)(\partial+\alpha\lambda+\beta)-\bar{\delta}_1\varphi(\partial)(\partial+\bar{\alpha}\lambda+\bar{\beta})$ and $g(\partial,\lambda)=\delta_2\varphi(\partial+\lambda)\phi(\lambda)-\bar{\delta}_2\varphi(\partial)\bar{\phi}(\lambda)$ for some polynomial $\varphi$.
\end{lemma}
\begin{proof}
	Assume that (\ref{f33}) is a trivial extension, that is, there exists $\bar{v}'=k(\partial)v+l(\partial)\bar{v}\in E$, where $k(\partial),l(\partial)\in\mathbb{C}$ and $ l(\partial)\neq0$, such that \begin{align*}
	&A_\lambda \bar{v}'=\bar{\delta}_1(\partial+\bar{\alpha}\lambda+\bar{\beta})\bar{v}'=\bar{\delta}_1k(\partial)(\partial+\bar{\alpha}\lambda+\bar{\beta})v+\bar{\delta}_1l(\partial)(\partial+\bar{\alpha}\lambda+\bar{\beta})\bar{v},\\
	&B_\lambda \bar{v}'=\bar{\delta}_2\bar{\phi}(\lambda)\bar{v}'=\bar{\delta}_2k(\partial)\bar{\phi}(\lambda)v+\bar{\delta}_2l(\partial)\bar{\phi}(\lambda)\bar{v}.
	\end{align*}
	
	On the other hand, it follows from (\ref{se9}) that
	\begin{align*}
	&A_\lambda \bar{v}'=(\delta_1k(\partial+\lambda)(\partial+\alpha\lambda+\beta)+l(\partial+\lambda)f(\partial,\lambda))v+\bar{\delta}_1l(\partial+\lambda)(\partial+\bar{\alpha}\lambda+\bar{\beta})\bar{v},\\
	&B_\lambda \bar{v}'=(\delta_2k(\partial+\lambda)\phi(\lambda)+l(\partial+\lambda)g(\partial,\lambda))v+\bar{\delta}_2l(\partial+\lambda)\bar{\phi}(\lambda)\bar{v}.
	\end{align*}
	Comparing both expressions for $A_\lambda \bar{v}'$ and $B_\lambda \bar{v}'$, we can obtain that $l$ is a nonzero constant. And then we can give the expressions of $f(\partial,\lambda)$ and $g(\partial,\lambda)$.
	
	Conversely, if $f(\partial,\lambda)=\delta_1\varphi(\partial+\lambda)(\partial+\alpha\lambda+\beta)-\bar{\delta}_1\varphi(\partial)(\partial+\bar{\alpha}\lambda+\bar{\beta})$ and $g(\partial,\lambda)=\delta_2\varphi(\partial+\lambda)\phi(\lambda)-\bar{\delta}_2\varphi(\partial)\bar{\phi}(\lambda)$ for some polynomial $\varphi$, setting $\bar{v}'=-\varphi(\partial)v+\bar{v}$ we can deduce that (\ref{f33}) is a trivial extension.
\end{proof}

\begin{theorem}\label{t33}
	For rank two Lie conformal algebra $\mathcal{R}$ that is of Type I, nontrivial extensions of finite irreducible conformal modules of the form (\ref{f33}) exist only when $(\delta_1,\delta_2)=(\bar{\delta_1},\bar{\delta_2})$. Moreover, they are given (up to equivalence) by (\ref{se9}). The value of $\delta_i,\bar{\delta}_i,i=1,2,\alpha,\bar{\alpha},\beta,\bar{\beta},$ and the corresponding polynomials $\phi(\lambda),\bar{\phi}(\lambda),f(\partial,\lambda)$ and $g(\partial,\lambda)$ giving rise to nontrivial extensions, are listed as follows:
	\begin{enumerate}
		\item In the case that $(\delta_1,\delta_2)=(\bar{\delta_1},\bar{\delta_2})=(1,0)$, $g=0,\beta=\bar{\beta},\bar{\alpha}-\alpha\in\{0,1,2,3,4,5,6\},\alpha,\bar{\alpha}\neq0,$ and	\begin{enumerate}[(i)]
			\item $\bar{\alpha}=\alpha$, $f(\partial,\lambda)=s_0+s_1\lambda$, where $(s_0,s_1)\neq (0,0)$.
			\item $\bar{\alpha}-\alpha=2$, $f(\partial,\lambda)=s\lambda^2(2(\partial+\beta)+\lambda)$, where $s\neq0$.
			\item $\bar{\alpha}-\alpha=3$, $f(\partial,\lambda)=s(\partial+\beta)\lambda^2((\partial+\beta)+\lambda)$, where $s\neq0$.
			\item $\bar{\alpha}-\alpha=4$, $f(\partial,\lambda)=s\lambda^2(4(\partial+\beta)^3+6(\partial+\beta)^2\lambda-(\partial+\beta)\lambda^2+\alpha_1\lambda^3)$, where $s\neq0$.
			
			\item $\bar{\alpha}=1$ and $\alpha=-4$, $f(\partial,\lambda)=s((\partial+\beta)^4\lambda^2-10(\partial+\beta)^2\lambda^4-17(\partial+\beta)\lambda^5-8\lambda^6)$, where $s\neq0$.
			
			\item $\bar{\alpha}=\frac{7}{2}\pm\frac{\sqrt{19}}{2}$ and $\alpha=-\frac{5}{2}\pm\frac{\sqrt{19}}{2}$, $f(\partial,\lambda)=s((\partial+\beta)^4\lambda^3-(2\alpha+3)(\partial+\beta)^3\lambda^4-3\alpha(\partial+\beta)^2\lambda^5-(3\alpha+1)(\partial+\beta)\lambda^6-(\alpha+\frac{9}{28})\lambda^7)$, where $s\neq0$.
		\end{enumerate}
	The value of $dim(Ext(V_{\bar{\alpha},\bar{\beta},\bar{\phi}},V_{\alpha,\beta,\phi}))$ is 2 in subcase (i), and 1 in subcases (ii)-(vi).
		\item In the case that $(\delta_1,\delta_2)=(\bar{\delta_1},\bar{\delta_2})=(0,1)$,  $\phi(\lambda)=\bar{\phi}(\lambda),f(\partial,\lambda)=s(\lambda),g(\partial,\lambda)=t(\lambda)$ with polynomials $s,t$ and either $s(\lambda)\neq0$ or $t(\lambda)$ is not a scalar multiple of $\lambda\phi(\lambda)$. Then the space $Ext(V_{\bar{\alpha},\bar{\beta},\bar{\phi}},V_{\alpha,\beta,\phi})$ is infinite-dimensional.	
		
	\end{enumerate}
\end{theorem}
\begin{proof}
	Applying both sides of (\ref{m31}), (\ref{m32}) and (\ref{m33}) to $\bar{v}$ and comparing the corresponding coefficients, we obtain \begin{align}
	&\bar{\delta}_1f(\partial,\lambda)(\partial+\lambda+\bar{\alpha}\mu+\bar{\beta})+\delta_1f(\partial+\lambda,\mu)(\partial+\alpha\lambda+\beta)\notag\\
	&\qquad-\bar{\delta}_1f(\partial,\mu)(\partial+\mu+\bar{\alpha}\lambda+\bar{\beta})-\delta_1f(\partial+\mu,\lambda)(\partial+\alpha\mu+\beta)=(\lambda-\mu)f(\partial,\lambda+\mu),\label{te91}\\
	&\bar{\delta}_2f(\partial,\lambda)\bar{\phi}(\mu)+\delta_1g(\partial+\lambda,\mu)(\partial+\alpha\lambda+\beta)\notag\\
	&\qquad-\bar{\delta}_1g(\partial,\mu)(\partial+\mu+\bar{\alpha}\lambda+\bar{\beta})-\delta_2f(\partial+\mu,\lambda)\phi(\mu)=0,\label{te92}\\
	&\bar{\delta}_2g(\partial,\lambda)\bar{\phi}(\mu)+\delta_2g(\partial+\lambda,\mu)\phi(\lambda)-\bar{\delta}_2g(\partial,\mu)\bar{\phi}(\lambda)-\delta_2g(\partial+\mu,\lambda)\phi(\mu)=0.\label{te93}
	\end{align}
	
	If $(\delta_1,\bar{\delta}_1,\delta_2,\bar{\delta}_2)=(1,1,0,0)$, the result can be deduced from Lemma \ref{r3}, Theorem 3.2 in \cite{ckw1} (or Theorem 2.7 in \cite{lhw2}) and Lemma \ref{l33}.
	
	If $(\delta_1,\bar{\delta}_1,\delta_2,\bar{\delta}_2)=(0,0,1,1)$, then (\ref{te92}) implies $f(\partial,\lambda)=0$ if $\phi(\lambda)\neq\bar{\phi}(\lambda)$ and $f(\partial,\lambda)=s(\lambda)$ for some polynomial $s$ if $\phi(\lambda)=\bar{\phi}(\lambda)$. Applying Lemma \ref{r2} to (\ref{te93}), we have $g(\partial,\lambda)=t(\partial+\lambda)\phi(\lambda)-t(\partial)\bar{\phi}(\lambda)$ for some polynomial $t$ if $\phi(\lambda)\neq\bar{\phi}(\lambda)$ and $g(\partial,\lambda)=(t_1(\partial+\lambda)-t_1(\partial))\phi(\lambda)+t_2(\lambda)$ for some polynomial $t_1,t_2$ if $\phi(\lambda)=\bar{\phi}(\lambda)$. By Lemma \ref{l33}, the extension is nontrivial only when $\phi=\bar{\phi}$ and either $s\neq0$ or $t_2(\lambda)$ is not a scalar multiple of $\lambda \phi(\lambda)$.
	
	If $(\delta_1,\bar{\delta}_1,\delta_2,\bar{\delta}_2)=(1,0,0,1)$, then putting $\mu=0$ in (\ref{te91}), we can obtain $$f(\partial+\lambda,0)(\partial+\alpha\lambda+\beta)=f(\partial,\lambda)(\partial+\lambda+\beta).$$ So when $\alpha=1$, we have $f(\partial,\lambda)=f(\partial+\lambda,0)=s(\partial+\lambda)$ for some polynomial $s$. If $\alpha\neq1$, then one can deduce that $f(\partial,\lambda)=s(\partial+\lambda)(\partial+\alpha\lambda+\beta)$ for some polynomial $s$. Applying Lemma \ref{r2} to (\ref{te93}), we have $g(\partial,\lambda)=t(\partial)\bar{\phi}(\lambda)
$, where $t$ is a polynomial. Putting these results in (\ref{te92}), we can obtain\begin{align}
	\begin{cases}
	s(\partial+\lambda)+t(\partial+\lambda)(\partial+\lambda+\beta)=0,&\alpha=1,\\
	s(\partial+\lambda)+t(\partial+\lambda)=0,&\alpha\neq1,\\
	\end{cases}	
	\end{align}
	The solutions are concluded as follows.
	\begin{enumerate}[(i)]
		\item If $\alpha=1$, then $f(\partial,\lambda)=-t(\partial+\lambda)(\partial+\lambda+\beta),g(\partial,\lambda)=t(\partial)\bar{\phi}(\lambda)$ for some polynomial $t$. The extension is trivial.
		\item If $\alpha\neq1$, then $f(\partial,\lambda)=s(\partial+\lambda)(\partial+\alpha\lambda+\beta),g(\partial,\lambda)=-s(\partial)\bar{\phi}(\lambda)$ for some polynomial $s$. The extension is trivial.
	\end{enumerate}
	
	If $(\delta_1,\bar{\delta}_1,\delta_2,\bar{\delta}_2)=(0,1,1,0)$, one can deduce the result similarly.
	
\end{proof}
\section{For rank two Lie conformal algebras that are of Type II}\label{s6}
	In this section, we investigate the extension problems under the condition that $\mathcal{R}$ is the conformal algebra defined in (\ref{a3}). Then there is a basis $\{A,B\}$ such that
\begin{equation}
[A_\lambda A]=(\partial+2\lambda)A+Q(\partial,\lambda)B,\ [A_\lambda B]=(\partial+a\lambda+b)B,\ [B_\lambda B]=0.
\end{equation}When $Q(\partial,\lambda)=0$, $\mathcal{R}$ is a $\mathcal{W}(a,b)$ algebra, which had been discussed in \cite{lhw2}. So we only consider the case that $Q(\partial,\lambda)\neq0$, which means $b=0$ and $a\in\{1,0,-1,-4,-6\}$.
If $V$ is a non-trivial finite irreducible $\mathcal{R}$-module, then \begin{equation}
V\cong {V}_{\alpha,\beta}=\mathbb{C}[\partial]v,\quad A_\lambda v=(\partial+\alpha\lambda+\beta)v,\quad B_\lambda v=0,
\end{equation}
where $\beta,0\neq\alpha\in\mathbb{C}$.  

By definition \ref{d2}, the $\mathcal{R}$-module structure on ${V}_{\alpha,\beta}$ given by $A_\lambda,B_\lambda\in End_\mathbb{C}(V)[\lambda]$ satisfies
\begin{align}
&[A_{\lambda},A_{\mu}]=(\lambda-\mu)A_{\lambda+\mu}+Q(-\lambda-\mu,\lambda)B_{\lambda+\mu}, \label{m41}\\
&[A_{\lambda},B_{\mu}]=((a-1)\lambda-\mu)B_{\lambda+\mu}, \label{m42}\\
&[B_{\lambda},B_{\mu}]=0, \label{m43}\\
&[\partial,A_{\lambda}]=-\lambda A_{\lambda}, \label{m44}\\
&[\partial,B_{\lambda}]=-\lambda B_{\lambda}. \label{m45}
\end{align}

\subsection{$0\longrightarrow \mathbb{C}c_{\eta} \longrightarrow  E\longrightarrow  {V}_{\alpha,\beta} \longrightarrow   0$}\hfill\\
	
First, we consider extensions of finite irreducible $\mathcal{R}$-modules of the form \begin{align}\label{f41}
0\longrightarrow \mathbb{C}c_{\eta} \longrightarrow  E\longrightarrow  {V}_{\alpha,\beta} \longrightarrow   0.
\end{align}Then $E$ is isomorphic to $\mathbb{C}c_{\eta}\oplus {V}_{\alpha,\beta}=\mathbb{C}c_{\eta}\oplus \mathbb{C}[\partial]v$ as a vector space, and the following identities hold in $E$:
\begin{align}\label{se10}
&\mathcal{R}_\lambda c_{\eta}=0,\quad\partial c_{\eta}=\eta c_{\eta}, \notag\\
&A_\lambda v=(\partial+\alpha\lambda+\beta) v+f(\lambda)c_{\eta},\quad B_\lambda v=g(\lambda)c_{\eta}, 
\end{align}
where $f(\lambda), g(\lambda) \in \mathbb{C}[\lambda]$.

\begin{lemma}\label{l41}
	All trivial extensions of finite irreducible $\mathcal{R}$-modules of the form (\ref{f41}) are given by (\ref{se10}), and $f(\lambda)$ is a scalar multiple of $\alpha\lambda+\eta+\beta,g(\lambda)=0$.
	
\end{lemma}
\begin{proof}
	Assume that (\ref{f41}) is a trivial extension, that is, there exists $v'=kc_\eta+l(\partial)v\in E$, where $k\in\mathbb{C}$ and $0\neq l(\partial)\in\mathbb{C}[\partial]$, such that \begin{align*}
	&A_\lambda v'=(\partial+\alpha\lambda+\beta)v'=k(\eta+\alpha\lambda+\beta)c_\eta+l(\partial)(\partial+\alpha\lambda+\beta)v,\quad B_\lambda v'=0.
	\end{align*}
	
	On the other hand, it follows from (\ref{se10}) that
	\begin{align*}
	&A_\lambda v'=f(\lambda)l(\eta+\lambda)c_\eta+ l(\partial+\lambda)(\partial+\alpha\lambda+\beta)v,\\
	&B_\lambda v'=g(\lambda)l(\eta+\lambda)c_\eta.
	\end{align*}
	We can obtain that $l(\partial)$ is a nonzero constant and $g=0$ by comparing both expressions for $A_\lambda v'$ and $B_\lambda v'$. Thus $f(\lambda)$ is a scalar multiple of $\alpha\lambda+\eta+\beta$.
	
	Conversely, if $f(\lambda)=k(\alpha\lambda+\eta+\beta)$ and $g(\lambda)=0$ for some $k\in\mathbb{C}$, setting $v'=kc_\eta +v$ we can deduce that (\ref{f41}) is a trivial extension.

\end{proof}

\begin{theorem}\label{t41}
	For rank two Lie conformal algebra $\mathcal{R}$ that is of Type II with $Q\neq0$, nontrivial extensions of finite irreducible conformal modules of the form (\ref{f41}) exist only if $\beta+\eta=0$. Moreover, they are given (up to equivalence) by (\ref{se10}). The values of $\alpha$ along with the corresponding polynomials $f(\lambda)$ and $g(\lambda)$ giving rise to nontrivial extensions, are listed as follows:
	 $g(\lambda)=0$ and \begin{align*}
		f(\lambda)=\begin{cases}
		s_1\lambda^2,&\alpha=1,\\
		s_2\lambda^3,&\alpha=2,
		\end{cases}
		\end{align*}
		with nonzero constants $s_1,s_2$. In these cases, $dim(Ext({V}_{\alpha,\beta},\mathbb{C}c_\eta))=1$.	
\end{theorem}
\begin{proof}
	Applying both sides of (\ref{m41}) and (\ref{m42}) to $v$ and comparing the corresponding efficients, we obtain \begin{align}
	&(\eta+\lambda+\alpha\mu+\beta)f(\lambda)-(\eta+\mu+\alpha\lambda+\beta)f(\mu)=(\lambda-\mu)f(\lambda+\mu)+Q(-\lambda-\mu,\lambda)g(\lambda+\mu),\label{te101}\\
	&-(\eta+\mu+\alpha\lambda+\beta)g(\mu)=((a-1)\lambda-\mu)g(\lambda+\mu).\label{te102}
	\end{align}
	Setting $\lambda=0$ in (\ref{te102}) gives $$(\eta+\beta)g(\mu)=0.$$
	
	If $\beta+\eta\neq0$, then $g=0$. Putting $\mu=0$ in (\ref{te101}) and combining Lemma \ref{l41}, one can deduce the extension is trivial. 
	
	Assume $\beta+\eta=0$. If $g=0$, then one can obtain the result by Proposition 2.1 in \cite{ckw1}. Now we consider that $g\neq0$. If $a=1$, then (\ref{te102}) turns into $$(\alpha\lambda+\mu)g(\mu)=\mu g(\lambda+\mu),$$ which implies $\alpha\in\{0,1\}$ and $g(\lambda)=\begin{cases}
	t,&\alpha=0,\\
	t\lambda,&\alpha=1
	\end{cases}$ for some nonzero constant $t$. On the other hand, under the condition that $a=1$, we have $Q(\partial,\lambda)=c(\partial+2\lambda),c\neq0$. So (\ref{te101}) is equivalent to the equation $$(\lambda+\alpha\mu)f(\lambda)-(\mu+\alpha\lambda)f(\mu)=(\lambda-\mu)f(\lambda+\mu)+c(\lambda-\mu)g(\lambda+\mu).$$ Taking the value of $\alpha,g(\lambda)$ and $\mu=0$ in the variant, one can get a contradiction. So $g(\lambda)=0$ if $a=1$.
	
	If $\beta+\eta=0,g(\lambda)\neq0,a\neq1$, then (\ref{te102}) implies $\alpha=1-a$ and $g(\lambda)=t\in\mathbb{C}^\times$. So if $a=0$, then $\alpha=1$ and $Q(\partial,\lambda)=c\lambda(\partial+\lambda)(\partial+2\lambda)+d\partial(\partial+2\lambda)$. Putting $\mu=0$ in (\ref{te101}), we have $-\lambda f(0)=-dt\lambda^2$. Thus $d=0$ and $f(0)=0$. Then putting $\mu=-\lambda$ in (\ref{te101}), we can see that $0=2ct\lambda^3$, that is,  $c=0$, which contracts with $Q(\partial,\lambda)\neq0$. Hence $a\neq0$. Similarly, one can check that $a\neq-1$ by putting  $\mu=0,-\lambda,-2\lambda$ in (\ref{te101}) one after another. For $a=-4$, we can assume $f(\lambda)=s\lambda^6$. By comparing the coefficient of $\lambda^5\mu^2$ in (\ref{te101}), we can deduce that $s=c=0$. Thus, $Q(\partial,\lambda)=0$. For $a=-6$, we can assume $f(\lambda)=s\lambda^8$ and we can get a contradiction by comparing the coefficient of $\lambda^6\mu^3$ in (\ref{te101}).
	
\end{proof}
\subsection{$0\longrightarrow {V}_{\alpha,\beta} \longrightarrow  E\longrightarrow\mathbb{C}c_{\eta}   \longrightarrow   0$}\hfill\\

Next, we consider extensions of finite irreducible $\mathcal{R}$-modules of the form \begin{align}\label{f42}
0\longrightarrow {V}_{\alpha,\beta} \longrightarrow  E\longrightarrow\mathbb{C}c_{\eta}   \longrightarrow   0.
\end{align}Then $E$ is isomorphic to ${V}_{\alpha,\beta}\oplus\mathbb{C}c_{\eta} =\mathbb{C}[\partial]v\oplus\mathbb{C}c_{\eta} $ as a vector space, and the following identities hold in $E$:
\begin{align}\label{se11}
&A_\lambda v=(\partial+\alpha\lambda+\beta)v,\quad B_\lambda v=0, \notag\\
&A_\lambda c_{\eta}=f(\partial,\lambda)v,\quad B_\lambda c_{\eta}=g(\partial,\lambda)v,\quad\partial c_{\eta}=\eta c_{\eta}+h(\partial)v, 
\end{align}
where $f(\partial,\lambda), g(\partial,\lambda), h(\partial)\in \mathbb{C}[\partial,\lambda]$.

\begin{lemma}\label{l42}
	All trivial extensions of finite irreducible $\mathcal{R}$-modules of the form (\ref{f42}) are given by (\ref{se11}), and $f(\partial,\lambda)=\varphi(\partial+\lambda)(\partial+\alpha\lambda+\beta),g(\partial,\lambda)=0$ and $h(\partial)=(\partial-\eta)\varphi(\partial)$, where $\varphi$ is a polynomial.
\end{lemma}
\begin{proof}
	Assume that (\ref{f42}) is a trivial extension, that is, there exists $c'_\eta=kc_\eta+l(\partial)v\in E$, where $0\neq k\in\mathbb{C}$ and $ l(\partial)\in\mathbb{C}[\partial]$, such that $A_\lambda c'_\eta=B_\lambda c'_\eta=0$ and $\partial c'_{\eta}=\eta c'_{\eta}=k\eta c_\eta+\eta l(\partial)v$.
	
	On the other hand, it follows from (\ref{se11}) that
	\begin{align*}
	&A_\lambda c'_\eta=(kf(\partial,\lambda)+l(\partial+\lambda)(\partial+\alpha\lambda+\beta))v, \quad B_\lambda c'_\eta=kg(\partial,\lambda)v,\\
	&\partial c'_{\eta}=k\eta c_\eta+(kh(\partial)+\partial l(\partial))v.
	\end{align*}
	We can obtain the result by comparing both expressions for  $A_\lambda c'_\eta,B_\lambda c'_\eta$ and $\partial c'_\eta$.
	
	Conversely, if $f(\partial,\lambda)=\varphi(\partial+\lambda)(\partial+\alpha\lambda+\beta),g(\partial,\lambda)=0$ and $h(\partial)=(\partial-\eta)\varphi(\partial)$ for some polynomial $\varphi$, setting $c'_{\eta}=c_\eta -\varphi(\partial)v$, we can deduce that (\ref{f42}) is a trivial extension. 
\end{proof}

\begin{theorem}\label{t42}
	For rank two Lie conformal algebra $\mathcal{R}$ that is of Type II, nontrivial extensions of finite irreducible conformal modules of the form (\ref{f42}) exist only if $\beta+\eta=0$ and $\alpha=1$. Moreover, they are given (up to equivalence) by (\ref{se11}) and $dim(Ext(\mathbb{C}c_\eta,{V}_{\alpha,\beta}))=1$. The corresponding polynomials $f(\partial,\lambda),g(\partial,\lambda)$ and $h(\partial)$ giving rise to nontrivial extensions, are listed as follows: $g(\partial,\lambda)=0$ and $f(\partial,\lambda)=h(\partial)=s$ with nonzero constant $s$.		
\end{theorem}
\begin{proof}
	Applying both sides of (\ref{m44}) and (\ref{m45})  to $c_\eta$ and comparing the corresponding efficients gives the following equations \begin{align}
	&(\partial+\lambda-\eta)f(\partial,\lambda)=h(\partial+\lambda)(\partial+\alpha\lambda+\beta),\label{te111}\\
	&(\partial+\lambda-\eta)g(\partial,\lambda)=0.\label{te112}
	\end{align}
	
	Then $g(\partial,\lambda)=0$ by (\ref{te112}), and the result can be deduced by Proposition 2.2 in \cite{ckw1}. 
\end{proof}
\subsection{$0\longrightarrow {V}_{\alpha,\beta} \longrightarrow  E\longrightarrow  {V}_{\bar{\alpha},\bar{\beta}} \longrightarrow   0$}\hfill\\
Finally, we consider extensions of finite irreducible $\mathcal{R}$-modules of the form \begin{align}\label{f43}
0\longrightarrow {V}_{\alpha,\beta} \longrightarrow  E\longrightarrow  {V}_{\bar{\alpha},\bar{\beta}} \longrightarrow   0.
\end{align}Then $E$ is isomorphic to ${V}_{\alpha,\beta} \oplus  {V}_{\bar{\alpha},\bar{\beta}} =\mathbb{C}[\partial]v\oplus\mathbb{C}[\partial]\bar{v} $ as a vector space, and the following identities hold in $E$:
\begin{align}\label{se12}
&A_\lambda v=(\partial+\alpha\lambda+\beta)v,\quad B_\lambda v=0, \notag\\
&A_\lambda \bar{v}=(\partial+\bar{\alpha}\lambda+\bar{\beta})\bar{v}+f(\partial,\lambda)v,\quad B_\lambda \bar{v}=g(\partial,\lambda)v, 
\end{align}
where $f(\partial,\lambda), g(\partial,\lambda)\in \mathbb{C}[\partial,\lambda]$.
\begin{lemma}\label{l43}
	All trivial extensions of finite irreducible $\mathcal{R}$-modules of the form (\ref{f43}) are given by (\ref{se12}), and $f(\partial,\lambda)=\varphi(\partial+\lambda)(\partial+\alpha\lambda+\beta)-\varphi(\partial)(\partial+\bar{\alpha}\lambda+\bar{\beta})$ and $g(\partial,\lambda)=0$ for some polynomial $\varphi$.
\end{lemma}
\begin{proof}
	Assume that (\ref{f43}) is a trivial extension, that is, there exists $\bar{v}'=k(\partial)v+l(\partial)\bar{v}\in E$, where $k(\partial),l(\partial)\in\mathbb{C}$ and $ l(\partial)\neq0$, such that \begin{align*}
	&A_\lambda \bar{v}'=(\partial+\bar{\alpha}\lambda+\bar{\beta})\bar{v}'=k(\partial)(\partial+\bar{\alpha}\lambda+\bar{\beta})v+l(\partial)(\partial+\bar{\alpha}\lambda+\bar{\beta})\bar{v},\  B_\lambda \bar{v}'=0.
	\end{align*}
	
	On the other hand, it follows from (\ref{se12}) that
	\begin{align*}
	&A_\lambda \bar{v}'=(k(\partial+\lambda)(\partial+\alpha\lambda+\beta)+l(\partial+\lambda)f(\partial,\lambda))v+l(\partial+\lambda)(\partial+\bar{\alpha}\lambda+\bar{\beta})\bar{v},\\
	&B_\lambda \bar{v}'=l(\partial+\lambda)g(\partial,\lambda)v.
	\end{align*}
	Comparing both expressions for $A_\lambda \bar{v}'$ and $B_\lambda \bar{v}'$, we can obtain that $l$ is a nonzero constant. And then we can give the expressions of $f(\partial,\lambda)$ and $g(\partial,\lambda)$.
	
	Conversely, if $f(\partial,\lambda)=\varphi(\partial+\lambda)(\partial+\alpha\lambda+\beta)-\varphi(\partial)(\partial+\bar{\alpha}\lambda+\bar{\beta})$ and $g(\partial,\lambda)=0$ for some polynomial $\varphi$, setting $\bar{v}'=-\varphi(\partial)v+\bar{v}$, we can deduce that (\ref{f43}) is a trivial extension.
\end{proof}

\begin{theorem}\label{t43}
	For rank two Lie conformal algebra $\mathcal{R}$ that is of Type II, nontrivial extensions of finite irreducible conformal modules of the form (\ref{f43}) exist only if $\beta=\bar{\beta}$. Moreover, they are given (up to equivalence) by (\ref{se12}). The value of $\alpha,\bar{\alpha}$, and the corresponding polynomials $f(\partial,\lambda)$ and $g(\partial,\lambda)$ giving rise to nontrivial extensions, are listed as follows (by replacing $\partial$ by $\partial+\beta$):
	\begin{enumerate}
		\item In the case that $a=1$, $\bar{\alpha}-\alpha\in\{0,1,2,3,4,5,6\},\alpha,\bar{\alpha}\neq0$ and	\begin{enumerate}[(i)]
				\item $\bar{\alpha}=\alpha$, $f(\partial,\lambda)=s_0+s_1\lambda,g(\partial,\lambda)=0$, where $(s_0,s_1)\neq (0,0)$.
				\item $\bar{\alpha}-\alpha=1$, $f(\partial,\lambda)=\frac{ct}{\alpha}\partial,g(\partial,\lambda)=t\lambda$, where $t\neq 0$.
			\item $\bar{\alpha}-\alpha=2$ with $\alpha \neq-1$, $f(\partial,\lambda)=s\lambda^2(2\partial+\lambda),g(\partial,\lambda)=0$, where $s\neq0$.
			\item $\bar{\alpha}=1$ and $\alpha=-1$, $f(\partial,\lambda)=s\lambda^2(2\partial+\lambda)-ct(\partial^2-\lambda^2),g(\partial,\lambda)=t(\partial\lambda+\lambda^2)$, where $(s,t)\neq(0,0)$.
			\item $\bar{\alpha}-\alpha=3$, $f(\partial,\lambda)=s\partial\lambda^2(\partial+\lambda),g(\partial,\lambda)=0$, where $s\neq0$.
			\item $\bar{\alpha}-\alpha=4$, $f(\partial,\lambda)=s\lambda^2(4\partial^3+6\partial^2\lambda-\partial\lambda^2+\alpha_1\lambda^3),g(\partial,\lambda)=0$, where $s\neq0$.
			\item $\bar{\alpha}=1$ and $\alpha=-4$, $f(\partial,\lambda)=s(\partial^4\lambda^2-10\partial^2\lambda^4-17\partial\lambda^5-8\lambda^6),g(\partial,\lambda)=0$, where $s\neq0$.
			
			\item $\bar{\alpha}=\frac{7}{2}\pm\frac{\sqrt{19}}{2}$ and $\alpha=-\frac{5}{2}\pm\frac{\sqrt{19}}{2}$, $f(\partial,\lambda)=s(\partial^4\lambda^3-(2\alpha+3)\partial^3\lambda^4-3\alpha\partial^2\lambda^5-(3\alpha+1)\partial\lambda^6-(\alpha+\frac{9}{28})\lambda^7),g(\partial,\lambda)=0$, where $s\neq0$.
		\end{enumerate}
	The value of $dim(Ext({V}_{\bar{\alpha},\bar{\beta}},{V}_{\alpha,\beta}))$ is 2 in subcase (i) and (iv), and 1 in the other subcases.
		\item In the case that $a=0$, $\bar{\alpha}-\alpha\in\{0,1,2,3,4,5,6\},\alpha,\bar{\alpha}\neq0$ and	\begin{enumerate}[(i)]
			\item $\bar{\alpha}=\alpha$, $f(\partial,\lambda)=s_0+s_1\lambda,g(\partial,\lambda)=0$, where $(s_0,s_1)\neq (0,0)$.
			\item $\bar{\alpha}-\alpha=1$, $f(\partial,\lambda)=-\frac{ct}{\alpha}\partial\lambda-\frac{dt}{\alpha}\partial,g(\partial,\lambda)=t$, where $t\neq 0$.
			\item $\bar{\alpha}-\alpha=2$ with $\alpha \neq-1$, $f(\partial,\lambda)=s\lambda^2(2\partial+\lambda),g(\partial,\lambda)=0$, where $s\neq0$.
			\item $\bar{\alpha}=1$ and $\alpha=-1$, $f(\partial,\lambda)=s\lambda^2(2\partial+\lambda)+ct\partial^2\lambda+dt(\partial^2-\lambda^2),g(\partial,\lambda)=t(\partial+\lambda)$, where $(s,t)\neq(0,0)$.
			\item $\bar{\alpha}-\alpha=3$, $f(\partial,\lambda)=s\partial\lambda^2(\partial+\lambda),g(\partial,\lambda)=0$, where $s\neq0$.
			\item $\bar{\alpha}-\alpha=4$, $f(\partial,\lambda)=s\lambda^2(4\partial^3+6\partial^2\lambda-\partial\lambda^2+\alpha_1\lambda^3),g(\partial,\lambda)=0$, where $s\neq0$.
            \item $\bar{\alpha}=1$ and $\alpha=-4$, $f(\partial,\lambda)=s(\partial^4\lambda^2-10\partial^2\lambda^4-17\partial\lambda^5-8\lambda^6),g(\partial,\lambda)=0$, where $s\neq0$.
			
			\item $\bar{\alpha}=\frac{7}{2}\pm\frac{\sqrt{19}}{2}$ and $\alpha=-\frac{5}{2}\pm\frac{\sqrt{19}}{2}$, $f(\partial,\lambda)=s(\partial^4\lambda^3-(2\alpha+3)\partial^3\lambda^4-3\alpha\partial^2\lambda^5-(3\alpha+1)\partial\lambda^6-(\alpha+\frac{9}{28})\lambda^7),g(\partial,\lambda)=0$, where $s\neq0$.
		\end{enumerate}	
	The value of $dim(Ext({V}_{\bar{\alpha},\bar{\beta}},{V}_{\alpha,\beta}))$ is 2 in subcase (i) and (iv), and 1 in the other subcases.
		\item In the case that $a=-1$, $\bar{\alpha}-\alpha\in\{0,1,2,3,4,5,6\},\alpha,\bar{\alpha}\neq0$ and\begin{enumerate}[(i)]
			\item $\bar{\alpha}=\alpha$, $f(\partial,\lambda)=s_0+s_1\lambda,g(\partial,\lambda)=0$, where $(s_0,s_1)\neq (0,0)$.
			\item $\bar{\alpha}-\alpha=2$ with $\alpha\neq-\frac{1}{2}$, $f(\partial,\lambda)=s\lambda^2(2\partial+\lambda),g(\partial,\lambda)=0$, where $s\neq0$.
			\item $\bar{\alpha}=\frac{3}{2}$ and $\alpha=-\frac{1}{2}$, $f(\partial,\lambda)=s\lambda^2(2\partial+\lambda)-2dt\partial^2\lambda-ct(2\partial^2-\lambda^2),g(\partial,\lambda)=t$, where $(s,t)\neq(0,0)$.
			\item $\bar{\alpha}-\alpha=3$ with $\alpha\neq-1$, $f(\partial,\lambda)=s\partial\lambda^2(\partial+\lambda),g(\partial,\lambda)=0$, where $s\neq0$.
			\item $\bar{\alpha}=2$ and $\alpha=-1$, $f(\partial,\lambda)=s\partial\lambda^2(\partial+\lambda)-\frac{dt}{4}(2\partial^3\lambda+\lambda^4)-\frac{ct}{2}(\partial^3-2\partial\lambda^2-2\lambda^3),g(\partial,\lambda)=t(\partial+\frac{1}{2}\lambda)$, where $(s,t)\neq(0,0)$.
			\item $\bar{\alpha}-\alpha=4$, $f(\partial,\lambda)=s\lambda^2(4\partial^3+6\partial^2\lambda-\partial\lambda^2+\alpha_1\lambda^3),g(\partial,\lambda)=0$, where $s\neq0$.
			
			\item $\bar{\alpha}=1$ and $\alpha=-4$, $f(\partial,\lambda)=s(\partial^4\lambda^2-10\partial^2\lambda^4-17\partial\lambda^5-8\lambda^6),g(\partial,\lambda)=0$, where $s\neq0$.
			
			\item $\bar{\alpha}=\frac{7}{2}\pm\frac{\sqrt{19}}{2}$ and $\alpha=-\frac{5}{2}\pm\frac{\sqrt{19}}{2}$, $f(\partial,\lambda)=s(\partial^4\lambda^3-(2\alpha+3)\partial^3\lambda^4-3\alpha\partial^2\lambda^5-(3\alpha+1)\partial\lambda^6-(\alpha+\frac{9}{28})\lambda^7),g(\partial,\lambda)=0$, where $s\neq0$.
		\end{enumerate}
	The value of $dim(Ext({V}_{\bar{\alpha},\bar{\beta}},{V}_{\alpha,\beta}))$ is 2 in subcase (i) and (v), and 1 in the other subcases.
		\item In the case that $a=-4$, $\bar{\alpha}-\alpha\in\{0,1,2,3,4,5,6,7\},\alpha,\bar{\alpha}\neq0,$ and\begin{enumerate}[(i)]
			\item $\bar{\alpha}=\alpha$, $f(\partial,\lambda)=s_0+s_1\lambda,g(\partial,\lambda)=0$, where $(s_0,s_1)\neq (0,0)$.
			\item $\bar{\alpha}-\alpha=2$, $f(\partial,\lambda)=s\lambda^2(2\partial+\lambda),g(\partial,\lambda)=0$, where $s\neq0$.
			\item $\bar{\alpha}-\alpha=3$, $f(\partial,\lambda)=s\partial\lambda^2(\partial+\lambda),g(\partial,\lambda)=0$, where $s\neq0$.
			\item $\bar{\alpha}-\alpha=4$, $f(\partial,\lambda)=s\lambda^2(4\partial^3+6\partial^2\lambda-\partial\lambda^2+\alpha_1\lambda^3),g(\partial,\lambda)=0$, where $s\neq0$.
			\item $\bar{\alpha}-\alpha=5$ with $\alpha\notin \{-2,-4\}$, $f(\partial,\lambda)=-\frac{3}{\alpha(\alpha+2)(\alpha+4)}ct\partial^3\lambda^3+\frac{9(\alpha+1)}{2\alpha(\alpha+2)(\alpha+4)}ct\partial^2\lambda^4-\frac{9(\alpha+1)(2\alpha+1)}{10\alpha(\alpha+2)(\alpha+4)}ct\partial\lambda^5+\frac{(\alpha+1)(2\alpha+1)}{10(\alpha+2)(\alpha+4)}ct\lambda^6,g(\partial,\lambda)=t$, where $t\neq0$.
			\item $\bar{\alpha}=3$ and $\alpha=-2$,  $f(\partial,\lambda)=\frac{3}{8}ct\partial^4\lambda^2-\frac{3}{2}ct\partial^2\lambda^4-\frac{57}{40}ct\partial\lambda^5-\frac{2}{5}ct\lambda^6,g(\partial,\lambda)=t$, where $t\neq0$.
			\item $\bar{\alpha}=1$ and $\alpha=-4$, $f(\partial,\lambda)=s(\partial^4\lambda^2-10\partial^2\lambda^4-17\partial\lambda^5-8\lambda^6),g(\partial,\lambda)=0$, where $s\neq0$.
			\item $\bar{\alpha}-\alpha=6$ with $\alpha\notin\{-\frac{5}{2},-\frac{5}{2}\pm\frac{\sqrt{19}}{2}\}$, $f(\partial,\lambda)=
			-\frac{3 }{(2\alpha+5)(2\alpha^2+10\alpha+3)}ct\partial^4\lambda^3
			+\frac{3 (2 \alpha  + 3 )}{(2\alpha+5)(2\alpha^2+10\alpha+3)}ct\partial^3\lambda^4
			-\frac{9(\alpha+1)(2\alpha+3)}{5(2\alpha+5)(2\alpha^2+10\alpha+3)}ct\partial^2\lambda^5
			+\frac{(\alpha+1)(2\alpha+1)(2\alpha+3)}{5(2\alpha+5)(2\alpha^2+10\alpha+3)}ct\partial\lambda^6-\frac{\alpha(\alpha+1)(2\alpha+1)(2\alpha+3)}{70(2\alpha+5)(2\alpha^2+10\alpha+3)}ct\lambda^7,g(\partial,\lambda)=t(\partial-\frac{\alpha}{5}\lambda)$, where $t\neq0$.
			\item $\bar{\alpha}=\frac{7}{2}$ and $\alpha=-\frac{5}{2}$, $f(\partial,\lambda)=\frac{36}{665}ct\bar{\partial}^5\lambda^2-\frac{54}{113}ct\bar{\partial}^3\lambda^4-\frac{387}{665}ct\bar{\partial}^2\lambda^5-\frac{218}{665}ct\bar{\partial}\lambda^6+\frac{127}{1862}ct\lambda^7,g(\partial,\lambda)=t(\partial+\frac{1}{2}\lambda)$, where $t\neq0$. 
			\item $\bar{\alpha}=\frac{7}{2}\pm\frac{\sqrt{19}}{2}$ and $\alpha=-\frac{5}{2}\pm\frac{\sqrt{19}}{2}$, $f(\partial,\lambda)=s(\partial^4\lambda^3-(2\alpha+3)\partial^3\lambda^4-3\alpha\partial^2\lambda^5-(3\alpha+1)\partial\lambda^6-(\alpha+\frac{9}{28})\lambda^7),g(\partial,\lambda)=0$, where $s\neq0$.
			\item $\bar{\alpha}=1$ and $\alpha=-6$,  $f(\partial,\lambda)=\frac{1}{35}ct\partial^5\lambda^3+\frac{2}{7}ct\partial^4\lambda^4+\frac{36}{35}ct\partial^3\lambda^5+\frac{12}{7}ct\partial^2\lambda^6+\frac{66}{49}ct\partial\lambda^7+\frac{99}{245}ct\lambda^8,g(\partial,\lambda)=t(\partial^2+\frac{11}{5}\partial\lambda+\frac{6}{5}\lambda^2)$, where $t\neq0$.
		\end{enumerate}
	The value of $dim(Ext({V}_{\bar{\alpha},\bar{\beta}},{V}_{\alpha,\beta}))$ is 2 in subcase (i), and 1 in the other subcases.
		\item In the case that $a=-6$, $\bar{\alpha}-\alpha\in\{0,1,2,3,4,5,6,7,8\},\alpha,\bar{\alpha}\neq0,$ and\begin{enumerate}[(i)]
			\item $\bar{\alpha}=\alpha$, $f(\partial,\lambda)=s_0+s_1\lambda,g(\partial,\lambda)=0$, where $(s_0,s_1)\neq (0,0)$.
			\item $\bar{\alpha}-\alpha=2$, $f(\partial,\lambda)=s\lambda^2(2\partial+\lambda),g(\partial,\lambda)=0$, where $s\neq0$.
			\item $\bar{\alpha}-\alpha=3$, $f(\partial,\lambda)=s\partial\lambda^2(\partial+\lambda),g(\partial,\lambda)=0$, where $s\neq0$.
			\item $\bar{\alpha}-\alpha=4$, $f(\partial,\lambda)=s\lambda^2(4\partial^3+6\partial^2\lambda-\partial\lambda^2+\alpha_1\lambda^3),g(\partial,\lambda)=0$, where $s\neq0$.
			\item $\bar{\alpha}=1$ and $\alpha=-4$, $f(\partial,\lambda)=s(\partial^4\lambda^2-10\partial^2\lambda^4-17\partial\lambda^5-8\lambda^6),g(\partial,\lambda)=0$, where $s\neq0$.
			
			\item $\bar{\alpha}=\frac{7}{2}\pm\frac{\sqrt{19}}{2}$ and $\alpha=-\frac{5}{2}\pm\frac{\sqrt{19}}{2}$, $f(\partial,\lambda)=s(\partial^4\lambda^3-(2\alpha+3)\partial^3\lambda^4-3\alpha\partial^2\lambda^5-(3\alpha+1)\partial\lambda^6-(\alpha+\frac{9}{28})\lambda^7),g(\partial,\lambda)=0$, where $s\neq0$.
			\item $\bar{\alpha}=4\pm\frac{\sqrt{22}}{2}$ and $\alpha=-3\pm\frac{\sqrt{22}}{2}$, $f(\partial,\lambda)=
		-\frac{40}{7(\alpha+3)}ct\partial^5\lambda^3
			+\frac{100(\alpha+2)}{7(\alpha+3)}ct\partial^4\lambda^4
			+\frac{40(5\alpha+1)}{7(\alpha+3)}ct\partial^3\lambda^5
			+\frac{20(16\alpha+11)}{7(\alpha+3)}ct\partial^2\lambda^6
			+\frac{10(154\alpha+101)}{49(\alpha+3)}ct\partial\lambda^7+\frac{823\alpha+539}{98(\alpha+3)}ct\lambda^8,g(\partial,\lambda)=t$, where $t\neq0$.
			\item $\bar{\alpha}=7$ and $\alpha=-1$, $f(\partial,\lambda)=-\frac{2}{7}ct\partial^6\lambda^3+\frac{9}{7}ct\partial^5\lambda^4-\frac{9}{7}ct\partial^4\lambda^5+\frac{2}{7}ct\partial^3\lambda^6,g(\partial,\lambda)=t(\partial+\frac{1}{7}\lambda)$, where $t\neq0$. 
			\item $\bar{\alpha}=2$ and  $\alpha=-6$, $f(\partial,\lambda)=-\frac{2}{7}ct\partial^6\lambda^3-3ct\partial^5\lambda^4-12ct\partial^4\lambda^5-24ct\partial^3\lambda^6-\frac{180}{7}ct\partial^2\lambda^7-\frac{99}{7}ct\partial\lambda^8-\frac{22}{7}ct\lambda^9,g(\partial,\lambda)=t(\partial+\frac{6}{7}\lambda)$, where $t\neq0$.
		\end{enumerate}
	The value of $dim(Ext({V}_{\bar{\alpha},\bar{\beta}},{V}_{\alpha,\beta}))$ is 2 in subcase (i), and 1 in the other subcases.
	\end{enumerate}
\end{theorem}
\begin{proof}
	Applying both sides of (\ref{m41}) and (\ref{m42}) to $\bar{v}$ and comparing the corresponding efficients, we obtain \begin{align}
	&f(\partial,\lambda)(\partial+\lambda+\bar{\alpha}\mu+\bar{\beta})+f(\partial+\lambda,\mu)(\partial+\alpha\lambda+\beta)-f(\partial,\mu)(\partial+\mu+\bar{\alpha}\lambda+\bar{\beta})\notag\\
	&\qquad -f(\partial+\mu,\lambda)(\partial+\alpha\mu+\beta)=(\lambda-\mu)f(\partial,\lambda+\mu)+Q(-\lambda-\mu,\lambda)g(\partial,\lambda+\mu),\label{te121}\\
	&g(\partial+\lambda,\mu)(\partial+\alpha\lambda+\beta)-g(\partial,\mu)(\partial+\mu+\bar{\alpha}\lambda+\bar{\beta})=((a-1)\lambda-\mu)g(\partial,\lambda+\mu).\label{te122}
	\end{align}
	
	Setting $\lambda=0$ in (\ref{te122}) gives $$g(\partial,\mu)(\beta-\bar{\beta})=0.$$
	
	If $g(\partial,\mu)=0$, then the result follows from Theorem 3.2 in \cite{ckw1} (or Theorem 2.7 in \cite{lhw2}). Now we assume $g(\partial,\lambda)\neq0$ so that $\beta=\bar{\beta}$. If no confusion is possible, we replace $\partial+\beta$ by $\partial$ in the sequel. By Proposition 3.8 and Corollary 3.10 in \cite{lhw2}, for $a\in\{0,-1,-4,-6\}$, the nonzero solutions (up to a nonzero scalar $t$) of (\ref{te122}) are given by 
	\begin{align}
g(\partial,\lambda)=\begin{cases}	
	1,&\alpha-\bar{\alpha}=a-1,\\
\partial-\frac{1}{1-a}\alpha\lambda,&\alpha-\bar{\alpha}=a-2,\\
\partial^2-\frac{1}{1-a}(1+2\alpha)\partial\lambda-\frac{1}{1-a}\alpha\lambda^2,&\alpha=a-2, \bar{\alpha}=1,
	\end{cases}
	\end{align} 
	and for $a=1$,\begin{align}
	g(\partial,\lambda)=\begin{cases}	
	1,&\alpha-\bar{\alpha}=0,\\
	\lambda,&\alpha-\bar{\alpha}=-1,\\
	\partial\lambda-\alpha\lambda^2,&\alpha-\bar{\alpha}=-2,\\
	\partial^2\lambda+3\partial\lambda^2+2\lambda^3,&\alpha=-2,\bar{\alpha}=1.
	\end{cases}
	\end{align} 
Putting these results in \ref{te121}, we can obtain the expression of $f(\partial,\lambda)$ as follows. 

(1) If $a=1,\alpha=\bar{\alpha}$, (\ref{te121}) is writen as \begin{align}\label{te1211}
	&f(\partial,\lambda)(\partial+\lambda+{\alpha}\mu)+f(\partial+\lambda,\mu)(\partial+\alpha\lambda)-f(\partial,\mu)(\partial+\mu+{\alpha}\lambda)-f(\partial+\mu,\lambda)(\partial+\alpha\mu)\notag\\
&\qquad =(\lambda-\mu)f(\partial,\lambda+\mu)+ct(\lambda-\mu).
\end{align}
By the nature of (\ref{te1211}), we may assume that a solution to (\ref{te1211}) is a homogeneous polynomial in $\partial$ and $\lambda$ of degree 0, that is, $f(\partial,\lambda)=s$ for some constant $s$. Taking it in $(\ref{te1211})$, we have $ct=0$ which means $g(\partial,\lambda)=0$ or $Q(\partial,\lambda)=0$. This contradiction illustrates that the equation has no solution in this case.

(2) If $a=1,\alpha-\bar{\alpha}=-1$, (\ref{te121}) is writen as \begin{align}\label{te1212}
&f(\partial,\lambda)(\partial+\lambda+\mu+{\alpha}\mu)+f(\partial+\lambda,\mu)(\partial+\alpha\lambda)-f(\partial,\mu)(\partial+\lambda+\mu+{\alpha}\lambda)\notag\\
&\qquad -f(\partial+\mu,\lambda)(\partial+\alpha\mu)=(\lambda-\mu)f(\partial,\lambda+\mu)+ct(\lambda+\mu)(\lambda-\mu).
\end{align}
By the nature of (\ref{te1212}), we may assume that a solution to (\ref{te1212}) is a homogeneous polynomial in $\partial$ and $\lambda$ of degree 1, that is, $f(\partial,\lambda)=s_1\partial+s_2\lambda$ for some constant $s_1,s_2$. Taking it in $(\ref{te1212})$, we have $s_1=\frac{ct}{\alpha}$ with $\alpha\neq0$ which means $f(\partial,\lambda)=\frac{ct}{\alpha}\partial+s_2\lambda$. For other homogeneous parts, one can refer to the case that $g(\partial,\lambda)=0$. 

We have a similar discussion on the remaining 14 cases. The final results follow from Lemma \ref{l43} and Theorem 3.2 in \cite{ckw1} (or Theorem 2.7 in \cite{lhw2}). 
\end{proof}
	

\begin{thebibliography}{CLOK}
	
	\bibitem [1]{bch} R. Biswal, A. Chakhar, X. He, \textsl{Classification of rank two Lie conformal algebras}, J. Ramanujan Math. Soc., 36 (2021), 203-219.
    
    \bibitem[2]{bdak}B. Bakalov, A.  D'Andrea, V. Kac, \textsl{Theory of finite pseudoalgebras}, Adv. Math., 162 (2001), 1-140.
    
    \bibitem[3]{bdk} A. Barakat, A. De Sole, V. Kac, \textsl{Poisson vertex algebras in the theory of Hamiltonian equations}, Japan. J. Math., 4 (2009), 141-252.
    
    \bibitem[4]{ck}
    S.-J. Cheng, V. Kac, \textsl{Conformal modules},
    Asian J. Math.,  1 (1997), 181-193.
    
    \bibitem[5]{ckw1}
    S.-J. Cheng, V. Kac, M. Wakimoto, \textsl{Extensions of conformal modules},
    In Topological Fields Theory, Primitive Forms and Related Topics, Progr. Math. Vol. 160, M. Kashiwara, et al. Eds. Boston: Birkh\"{a}user, (1998), 79-129.
    
    \bibitem[6]{ckw2}
    S.-J. Cheng, V. Kac, M. Wakimoto, \textsl{Extensions of Neveu-Schwarz conformal modules},
    J. Math. Phys., (2000), 2271-2294.
    
    \bibitem[7]{dk} A. D'Andrea, V. Kac, \textsl{Structure theory of finite conformal algebras}, Selecta Math., 4 (1998), 377-418.
    
    \bibitem[8]{en} O. Evnin, R. Nivesvivat, \textsl{Hidden symmetries of the Higgs oscillator and the conformal algebra}, J. Phys. A., 50 (2017), 015202.
    
    \bibitem[9]{l}
    N. Lam, \textsl{Extensions of modules over supercurrent conformal algebras}, Comm. Algebra, 29 (2001), 3061-3068.
    
    \bibitem[10]{lhw1}
    L. Luo, Y. Hong, Z. Wu, \textsl{Finite irreducible modules of Lie conformal algebras $\mathcal{W}(a,b)$ and some Schr\"{o}dinger-Virasoro type Lie conformal algebras}, Int. J. Math., 30 (2019), 1950026.
    
	\bibitem [11]{lhw2} L. Luo, Y. Hong, Z. Wu, \textsl{Extensions of finite irreducible modules of Lie conformal algebras  W(a,b)  and some Schr\"{o}dinger-Virasoro type Lie conformal algebras}, Comm. Algebra, 48 (2020), 4774-4795.
	
	\bibitem[12]{ly1} K. Ling, L. Yuan, \textsl{Extensions of modules over the Heisenberg-Virasoro conformal algebra}, Int.J. Math., 28 (2017), 1750036.
	\vspace{.1cm}
	
	\bibitem[13]{ly2}  K. Ling, L. Yuan, \textsl{Extensions of modules over a class of Lie conformal algebras $\mathcal{W}(b)$}, J. Algebra Appl., 18 (2019), 1950164.
	
	\bibitem[14]{k} V. Kac, \textsl{Vertex algebras for beginners}, University Lecture Series 10,2nd ed. Amer. Math. Soc., Providence, RI, 1998.
	
	\bibitem[15]{wx} W. Wang, C. Xia, \textsl{Structure of a class of rank two Lie conformal algebras}, Algebra Colloq., 28 (2021), 169-180.
  
    \bibitem [16]{xhw} M. Xu, Y. Hong, Z. Wu, \textsl{Finite irreducible conformal modules of rank two Lie conformal algebras}, J. Algebra Appl., 20 (2021) 2150145.
   
    \bibitem [17]{yl} L. Yuan, K. Ling, \textsl{Extensions of Schr\"{o}dinger-Virasoro conformal modules}, Comm. Algebra, 47 (2019), 2883-2903.

    \bibitem [18]{yw} L. Yuan, Y. Wang, \textsl{Conformal modules and their extensions of a Lie conformal algebra related to a 2-dimensional Novikov algebra}, arXiv:1907.02960.
   


	\end{thebibliography}
	\end{document}